\RequirePackage{fix-cm}
\documentclass[smallextended]{svjour3}       
\smartqed  
\usepackage{graphicx}
\begin{document}

\title{Existence of weak solutions to the three-dimensional density-dependent generalized
incompressible magnetohydrodynamic flows 
}


\author{Weiping Yan        
}


\institute{College of of Mathematics, Jilin University, Changchun
130012, P.R. China.\\\email{yan8441@126.com}.}

\date{Received: date / Accepted: date}

\maketitle

\begin{abstract}
In this paper we consider the equations of the unsteady viscous, incompressible, and heat
conducting magnetohydrodynamic flows in a bounded three-dimensional domain with Lipschitz boundary. By an approximation scheme and a weak convergence method,
the existence of a weak solution to the three-dimensional density dependent generalized incompressible magnetohydrodynamic equations with large data is obtained.
\end{abstract}

\section{Introduction}
\label{intro}
The study of the dynamics of biological fluid in the presence
of magnetic field is very useful in understanding
the bioengineering and medical technology. The development of magnetic devices for
cell separation, targeted transport of magnetic particles as drug carriers, magnetic wound or cancer tumor treatment causing
magnetic hyperthermia, reduction of bleeding during surgeries or provocation of occlusion of the feeding vessels of cancer
tumors and the development of magnetic tracers, as well are well-known applications in this domain of research \cite{Andra,Bar}.

Magnetohydrodynamic flow of a non-Newtonian fluid in a channel of slowly varying cross section
in the presence of a uniform transverse magnetic field was studied in \cite{Bh}. In the recent past, El-Shehawey et al. \cite{El} studied
an unsteady flow of blood as an electrically conducting, incompressible, elastico-viscous fluid in the presence of magnetic
field through a rigid circular pipe by considering the streaming blood as a non-Newtonian fluid in the axial direction only.

In the present paper, we consider the following system of the three-dimensional incompressible magnetohydrodynamic flows \cite{Bar,Kan,Sar}:
\begin{eqnarray}\label{E1-1}
\rho_t+\textbf{div}(\rho\textbf{u})=0,~~\textbf{div}\textbf{u}=0,
\end{eqnarray}
\begin{eqnarray}
\label{E1-2}
(\rho\textbf{u})_t+\textbf{div}(\rho\textbf{u}\otimes\textbf{u})+\nabla P=(\nabla\times\textbf{H})\times\textbf{H}+\textbf{div}\textbf{S}(\rho,\theta,\textbf{D}(\textbf{u})),
\end{eqnarray}
\begin{eqnarray}
\label{E1-2R}
\partial_t(\rho Q(\theta))-\textbf{div}(\textbf{q}(\rho,\theta,\nabla\theta))&+&\textbf{div}(\rho Q(\theta)\textbf{u})\nonumber\\
&&-\textbf{S}(\rho,\theta,\textbf{D}(\textbf{u})):\nabla\textbf{u}-\nu|\nabla\times\textbf{H}|^2=0,~~~
\end{eqnarray}
\begin{eqnarray}
\label{E1-4}
\textbf{H}_t-\nabla\times(\textbf{u}\times\textbf{H})=-\nabla\times(\nu\nabla\times\textbf{H}),~~\textbf{div}\textbf{H}=0,
\end{eqnarray}
where $\rho$, $\textbf{u}\in\textbf{R}^3$, $\textbf{H}\in\textbf{R}^3$ and $\theta$ denote the density, the velocity, the magnetic field and the
temperature, respectively; $\textbf{S}$ is the viscous stress tensor depending on the density, the temperature and the symmetric part of the velocity gradient $\textbf{D}(\textbf{u})$, the thermal
flux $\textbf{q}$ is a function of the density and the temperature, the $Q$ is a function of $\theta$.
The total energy given by
\begin{eqnarray*}
\Phi=\rho(e+\frac{1}{2}|\textbf{u}|^2)+\frac{1}{2}|\textbf{H}|^2,~~~~
\Phi'=\rho(e+\frac{1}{2}|\textbf{u}|^2),
\end{eqnarray*}
with the internal energy $e(\rho,\theta)$, the kinetic energy $\frac{1}{2}\rho|\textbf{u}|^2$, and the magnetic energy $\frac{1}{2}|\textbf{H}|^2$; $\textbf{D}(\textbf{u})=\nabla\textbf{u}+\nabla\textbf{u}^T$ is the symmetric part of the velocity gradient, $\nabla\textbf{u}^T$ is the transpose of the matrix $\nabla\textbf{u}$, and $\textbf{I}$ is the $3\times3$ identity matrix; $\nu> 0$ is the magnetic diffusivity acting as a magnetic diffusion coefficient of the magnetic field.

In order to have a clear picture about the admissible structure of these functions, we assume that $\textbf{S}$ and $\textbf{q}$ are of the form (for $\rho>0$, $\theta>0$, $\textbf{D}\in\textbf{R}^{3\times3}$ symmetric)
\begin{eqnarray}\label{E1-6R11}
\textbf{S}(\rho,\theta,\textbf{D})=2\mu_0(\rho,\theta,|\textbf{D}|^2)\textbf{D},~~~~\textbf{q}(\rho,\theta,\nabla\theta)=\kappa_0(\rho,\theta)\nabla\theta,
\end{eqnarray}
and behave as
\begin{equation}\label{E1-2'}
\left\{
\begin{array}{lll}
&&\textbf{S}(\rho,\theta,\textbf{D})\sim\mu(\rho,\theta)(\epsilon+|\textbf{D}|^2)^{\frac{r-2}{2}}\textbf{D},~~r\in(1,\infty),\\
&&\textbf{q}(\rho,\theta,\nabla\theta)\sim \kappa(\rho)\theta^{\alpha}\nabla\theta=\frac{\kappa(\rho)}{\alpha+1}\nabla\theta^{\alpha+1},~~\alpha\in\textbf{R},
\end{array}
\right.
\end{equation}
where $\epsilon\in[0,1]$, and there exist positive constants $\underline{\mu}$, $\overline{\mu}$, $\underline{\kappa}$ and $\overline{\kappa}$ such that
\begin{eqnarray*}
0<\underline{\mu}\leq\mu(\rho,\theta)\leq\overline{\mu}<+\infty,~~0<\underline{\kappa}\leq\kappa(\rho)\leq\overline{\kappa}<+\infty.
\end{eqnarray*}
Thus, in particular, for all $\rho>0$, $\theta>0$, and $\textbf{D}$, $\textbf{B}\in\textbf{R}^{3\times3}$ symmetric,
\begin{equation}\label{E1-1'}
\left\{
\begin{array}{lll}
&&\textbf{S}(\rho,\theta,\textbf{D})\cdot\textbf{D}\geq\underline{\mu}(\epsilon+|\textbf{D}|^2)^{\frac{r-2}{2}}|\textbf{D}|^2\geq0,\\
&&|\textbf{S}(\rho,\theta,\textbf{D})|\leq\overline{\mu}(\epsilon+|\textbf{D}|^2)^{\frac{r-2}{2}}|\textbf{D}|,\\
&&\left(\textbf{S}(\rho,\theta,\textbf{D})-\textbf{S}(\rho,\theta,\textbf{B})\right)\cdot\left(\textbf{D}-\textbf{B}\right)\geq0,
\end{array}
\right.
\end{equation}
and for all $\nabla\theta\in\textbf{R}^3(\rho\leq\rho^*)$,
\begin{equation}\label{E1-3'}
\left\{
\begin{array}{lll}
&&\textbf{q}(\rho,\theta,\nabla\theta)\cdot\nabla\theta\geq\underline{\kappa}\theta^{\alpha}|\nabla\theta|^2=\frac{4\underline{\kappa}}{(\alpha+2)^2}|\nabla\theta^{\frac{\alpha+2}{2}}|^2\geq0,\\
&&|\textbf{q}(\rho,\theta,\nabla\theta)|\leq\underline{\kappa}\theta^{\alpha}|\nabla\theta|.
\end{array}
\right.
\end{equation}
For simplicity, we impose the boundary conditions
\begin{eqnarray}\label{E1-11}
\textbf{u}=0,~~\textbf{H}=0,~~\textbf{q}\cdot\textbf{n}=0~~on~[0,T]\times\partial\Omega.
\end{eqnarray}
The initial density is supposed to be bounded and the initial total energy is integrable, i.e.,
\begin{equation}\label{E1-12}
\left\{
\begin{array}{lll}
&&\rho(0,\cdot)=\rho_0\in\textbf{L}^{\infty}(\Omega),\\
&&\rho(\frac{|\textbf{u}|^2}{2}+\theta)(0,\cdot)=\rho_0(\frac{|\textbf{u}_0|^2}{2}+\theta_0)\in\textbf{L}^1(\Omega),\\
&&\textbf{H}(0,\cdot)=\textbf{H}_0\in\textbf{L}^2(\Omega),
\end{array}
\right.
\end{equation}
and
\begin{eqnarray}\label{E1-13}
&&0<\rho_*\leq\rho_0(x)\leq\rho^*<+\infty~~a.a.~x\in\Omega,\\
\label{E1-14}
&&0<\theta_*\leq\theta_0(x)~~for~a.a.~x\in\Omega,
\end{eqnarray}
where $\rho_*$, $\rho^*$ and $\theta_*$ are constants.

The study of long time and large data existence theory for inhomogeneous incompressible fluids was investigated in several contributions. For the case $r=2$ and the viscosity does not depend on $|\textbf{D}(\textbf{u})|$, using the concept of renormalized solutions, Lions \cite{Lions0} established a new convergence and continuity properties of the density that may vanish at some parts of the domain where the viscosity depends on the density. Meanwhile, he got rid of the smallness of the data.
For the case that the viscosity depends on the shear rate ($r\neq2$), Fern\'{a}ndez-Cara et al. \cite{Fern} proved the existence of weak solutions for that $r\geq\frac{12}{5}$.  Guill\'{e}n-Gonz\'{a}lez \cite{G} (also see \cite{Lions0}) considered the spatially periodic setting by using higher differentiability method. Recently, Frehse et al.\cite{Fei1} established the existence result with non-slip boundary conditions and a viscosity that depends on both the density and the shear rate for $r>\frac{11}{5}$. In \cite{Freh2},  Frehse et al. showed the existence result on the full thermodynamic model for inhomogeneous incompressible fluids for $r\geq\frac{11}{5}$, which improved the result in \cite{Fei1}.
For more results about thermal flows of incompressible homogeneous fluids, we refer the reader to \cite{B1,B2,B3,Diening,Novo}.

Recently, there have been much work on magnetohydrodynamics because of its physical importance, complexity, and widely application (see \cite{Caban,Kuli,Laudau,Constan}).
Magnetohydrodynamics (MHD) is a combination of the compressible Navier-Stokes equations of fluid dynamics and Maxwell's equations of electromagnetism.
Duvaut and Lions \cite{Duv}, Sermange and Temam \cite{Ser} obtained some existence and long time
behavior results for incompressible case. For compressible magnetohydrodynamic flows of
Newtonian fluids, Ducomet and Feireisl \cite{Fei4} proved the existence of global in time weak solutions to
a multi-dimensional nonisentropic MHD system for gaseous stars coupled with the Poisson equation with
all the viscosity coefficients and the pressure depending on temperature and density asymptotically, respectively.
Hu and Wang \cite{Hu1} studied the global variational weak
solution to the three-dimensional full magnetohydrodynamic equations
with large data by an approximation scheme and a weak convergence
method. In \cite{Hu2}, by using the Faedo-Galerkin method and the
vanishing viscosity method, they also studied the existence and
large-time behavior of global weak solutions for the
three-dimensional equations of compressible magnetohydrodynamic
isentropic flows (\ref{E1-1})-(\ref{E1-2R}). They \cite{Hu3} showed that
the convergence of weak solutions of the compressible MHD system to a weak solution of the viscous
incompressible MHD system. Jiang, et all. \cite{Jiang,Jiang1}
obtained that the convergence towards the strong solution of the
ideal incompressible MHD system in the whole space and periodic
domain, respectively. For MHD driven by the time periodic external forces, Yan \cite{Yan} showed that such system has the time periodic weak solution. After that, Yan \cite{Yan1} obtained the weak-strong uniqueness property for full compressible magnetohydrodynamics flows.

The main difficulty of the study of MHD is the presence of the magnetic field and its interaction with the hydrodynamic motion in the MHD flow of large oscillation. This leads to that many fundamental problems for MHD are still open. For example, the global existence of classical solution to the full perfect MHD equations with large data in one dimensional case is unsolved.
But corresponding problem about Navier-Stokes equation was solved in \cite{Ka} a long time ago. In the present paper, we study the existence of weak solutions for the density-dependent generalized inhomogeneous incompressible Magnetohydrodynamic flows
in a bounded three-dimensional domain with Lipschitz boundary. Inspired by the work of \cite{Diperna,Freh2,Hu1,Lions1}, we will establish the existence of weak solutions for the density-dependent generalized inhomogeneous incompressible compressible MHD for any $r\geq\frac{11}{5}$.

These equations, and all functions involved in their descriptions as well, are considered in $(0,T)\times\Omega$, where
$\Omega\subset\textbf{R}^3$ is an open, connected and bounded set with Lipschitz boundary $\partial\Omega$, and $T\in(0,\infty)$.

The paper is organized in the following way: in the next section, by
introducing the appropriate function spaces we provide the precise definition of the
notion of weak solutions to system (\ref{E1-1})-(\ref{E1-4}). The main result of this paper is also stated.
Then in Section 3, we first introduce the corresponding approximation system whose solvability is established in Appendix (section 5). We also derive some corresponding uniform estimates. We finish the proof of Theorem 1 in section 4 by establishing the strongly convergence of $\{\rho^n\}$, $\{\textbf{u}^n\}$, $\{\theta^n\}$, $\{\textbf{H}^n\}$ and $\{\textbf{D}(\textbf{u}^n)\}$.

\section{Some notations and main result}
Before giving the definition of the weak solution to the problem (\ref{E1-1})-(\ref{E1-4}) with the boundary condition (\ref{E1-11}), we first state the following notation of relevant Banach spaces of functions defined on a bounded domain $\Omega\subset\textbf{R}^3$. For any $p\in[1,\infty]$, $\textbf{L}^p(\Omega)$ denotes the
Lebesgue spaces with the norm $\|\cdot\|_{\textbf{L}^p(\Omega)}$, $\textbf{W}^{1,p}(\Omega)$ denotes the
Sobolev spaces with the norm $\|\cdot\|_{\textbf{W}^{1,p}(\Omega)}$, $\textbf{W}_0^{1,p}(\Omega)$ denotes the closure of $\textbf{C}_0^{\infty}(\Omega)$ functions in the norm of $\textbf{W}^{1,p}(\Omega)$. If $X$ is a Banach space of scalar functions,
then $X^3$, $X^4$ or $X^{3\times3}$, $X^{4\times4}$ denote the space of vector or tensor-valued functions so
that each their component belongs to $X$. Further, we use the following notation for
the spaces of function with zero divergence and their dual $(r'=\frac{r}{r-1})$
\begin{eqnarray*}
&&\textbf{W}_{0,\textbf{div}}^{1,p}(\Omega):=\{\textbf{u}\in\textbf{W}^{1,p}_{0}(\Omega)^3;\textbf{div}\textbf{u}=0\},\\
&&\textbf{W}^{-1,q'}(\Omega)=(\textbf{W}_0^{1,q}(\Omega))^*,~~\textbf{W}^{-1,q'}_{\textbf{div}}(\Omega)=(\textbf{W}_{0,\textbf{div}}^{1,q}(\Omega))^*.
\end{eqnarray*}
$\textbf{L}_{\textbf{div}}^q(\Omega)$ denotes the closure of $\textbf{W}_{0,\textbf{div}}^{1,p}(\Omega)$ in $\textbf{L}^q(\Omega)^3$. The symbols $\textbf{L}^q(0,T;X)$ and $\textbf{C}(0,T;X)$ denote the standard Bochner spaces. We write $(a,b)$ instead of $\int_{\Omega}a(x)b(x)dx$ whenever $ab\in\textbf{L}^1(\Omega)$ and use the brackets $\langle a,b\rangle$ to denote the
duality pairing for $a\in X^*$ and $b\in X$. We use $\textbf{C}([0,T];\textbf{L}_{weak}^q(\Omega))$
to denote the space of
functions $\rho\in\textbf{L}^{\infty}(0,T;\textbf{L}^q(\Omega))$ satisfying $(\rho(t),z)\in\textbf{C}([0,T])$ for all $z\in\textbf{L}^{q'}$.
We do not distinguish between function spaces for scalar and vector valued
functions. Generic constants are denoted by $M$, their values may vary in the same
formula or in the same line.

\begin{definition}
Assume that $\textbf{S}$ and $\textbf{q}$ are continuous functions of the form $(\ref{E1-6R11})$ satisfying (\ref{E1-2'})-(\ref{E1-3'}) with
$r\geq\frac{11}{5}$ and $\alpha>-\frac{2}{3}$. The initial data $\rho_0$, $\textbf{u}_0$, $\theta_0$ and $\textbf{H}_0$ satisfy (\ref{E1-12})-(\ref{E1-14}).
A vector $(\rho,\textbf{u},\theta,\textbf{H})$ is said to be a weak solution to the problem (\ref{E1-1})-(\ref{E1-4}) of the generalized incompressible MHD equations if the following conditions hold:

$\bullet$ The density
\begin{eqnarray*}
\rho\in\textbf{C}([0,T];\textbf{L}^q(\Omega)),~~\rho_t\in\textbf{L}^{p_1}(0,T;(\textbf{W}^{1,r}(\Omega))^*),~~\forall q\in[1,\infty),
\end{eqnarray*}
and
\begin{eqnarray}\label{E2-2}
\int_{0}^T\langle\rho_t,z\rangle-(\rho\textbf{u},\nabla z)dt=0,
\end{eqnarray}
for any $z\in\textbf{L}^{r}(0,T;\textbf{W}^{1,r}(\Omega))$.

$\bullet$ The temperature non-negative $\theta$ function, the velocity function $\textbf{u}$ and the magnetic field $\textbf{H}$ satisfy
\begin{eqnarray*}
&&\textbf{u}\in\textbf{L}^{r}(0,T;\textbf{W}^{1,r}_{0,\textbf{div}}(\Omega)),\\
&&\partial_t(\rho\textbf{u})\in\textbf{L}^{r'}(0,T;\textbf{W}^{-1,r'}_{\textbf{div}}(\Omega)),\\
&&(\rho\textbf{u},\psi)\in\textbf{C}([0,T];\Omega)~~\forall\psi\in\textbf{L}^2_{\textbf{div}}(\Omega),\\
&&\theta\in\textbf{L}^{\infty}(0,T;\textbf{L}^1(\Omega)),\\
&&\theta^{\frac{\alpha-\lambda+1}{2}}\in\textbf{L}^{2}(0,T;\textbf{W}^{1,2}(\Omega)),\\
&&\partial_t(\rho\theta)\in\textbf{L}^1(0,T;(\textbf{W}^{1,q}(\Omega))^*)~~q~large~enough,\\
&&\textbf{H}\in\textbf{L}^{2}(0,T;\textbf{W}^{1,2}(\Omega)),~~\textbf{H}_t\in\textbf{L}^2(0,T;\textbf{W}^{-1,2}(\Omega)),
\end{eqnarray*}
and the following weak formulations hold:
\begin{eqnarray}\label{E2-R1}
\int_0^T(\langle(\rho\textbf{u})_t,\varphi\rangle-(\rho\textbf{u}\otimes\textbf{u},\nabla\varphi)&+&(\textbf{S}(\rho,\theta,\textbf{D}(\textbf{u})),\textbf{D}(\varphi))\nonumber\\
&&+\int_{\Omega}(\textbf{H}^T\nabla\varphi\textbf{H}+\frac{1}{2}\nabla(|\textbf{H}|^2)\cdot\varphi dx))dt=0,~~~~~
\end{eqnarray}
for all $\varphi\in\textbf{L}^{r}(0,T;\textbf{W}_{0,\textbf{div}}^{1,r}(\Omega))$,

\begin{eqnarray}\label{E2-R2}
\int_0^T(\langle(\rho Q(\theta))_t,h\rangle&-&(\rho Q(\theta)\textbf{u},\nabla h)-(\textbf{q}(\rho,\theta,\nabla\theta),\nabla h))dt\nonumber\\
&=&\int_0^T\left(\nu(|\nabla\times\textbf{H}|^2,h)+(\textbf{S}(\rho,\theta,\textbf{D}(\textbf{u})),\textbf{D}(\textbf{u})h)\right)dt,
\end{eqnarray}
for all $h\in\textbf{L}^{\infty}(0,T;\textbf{W}^{1,q}(\Omega))$ with $q$ sufficiently large,

\begin{eqnarray}\label{E2-R3}
\int_0^T(\langle\textbf{H}_t,b\rangle+\nu(\nabla\times\textbf{H},\nabla\times b)-(\textbf{u}\times\textbf{H},\nabla\times b))dt=0,
\end{eqnarray}
for all $b\in\textbf{L}^{2}(0,T;\textbf{W}^{1,2}(\Omega))$.

$\bullet$  The initial conditions are attained in the following sense
\begin{eqnarray*}
&&\lim_{t\longrightarrow0^+}\|\rho(t)-\rho_0\|_{\textbf{L}^q(\Omega)}+\|\textbf{u}(t)-\textbf{u}_0\|^2_{\textbf{L}^2(\Omega)}+\|\textbf{H}(t)-\textbf{H}_0\|_{\textbf{L}^{2}(\Omega)}=0,~\forall q\in[1,\infty),\\
&&\lim_{t\longrightarrow0^+}((\rho\theta)(t),h)=(\rho_0\theta_0,h),~~\forall h\in\textbf{L}^{\infty}(\Omega).
\end{eqnarray*}
\end{definition}
The aim of this paper is to establish the following result.
\begin{theorem}
Assume that $\textbf{S}$ and $\textbf{q}$ are continuous functions of the form $(\ref{E1-6R11})$ satisfying (\ref{E1-2'})-(\ref{E1-3'}) with
$r\geq\frac{11}{5}$ and $\alpha>-\frac{2}{3}$, and there are two positive constants $\underline{c_{\nu}}$ and $\overline{c_{\nu}}$ such that
\begin{eqnarray}\label{E2-1R}
0<\underline{c_{\nu}}\leq c_{\nu}(\theta)\leq\overline{c_{\nu}}<+\infty.
\end{eqnarray}
 Then there exists a weak solution to the problem (\ref{E1-1})-(\ref{E1-4}) in the sense of Definition 1 with initial data satisfying (\ref{E1-12})-(\ref{E1-14}).
\end{theorem}
Note that we consider in (\ref{E2-R1}) only divergenceless test function, so the pressure does not appear in the definition of weak solutions. The pressure cannot be a function of $\rho$ and $\theta$.
The pressure $P$ (zero mean value) can be obtained by comparing two auxiliary Stokes problems (homogeneous Dirichlet boundary conditions)
\begin{eqnarray*}
&&-\triangle\textbf{v}^1+\nabla P^1=\textbf{div}(\rho\textbf{u}\otimes\textbf{u}-\textbf{S}(\rho,\theta,\textbf{D}(\textbf{u}))-(\nabla\times\textbf{H})\times\textbf{H},\\
&&\textbf{div}\textbf{u}^1=0,
\end{eqnarray*}
and
\begin{eqnarray*}
&&-\triangle\textbf{v}^2+\nabla P^2=(\rho\textbf{u})(t)-\rho_0\textbf{u}_0\in\textbf{L}^2((0,T)\times\Omega)\hookrightarrow\textbf{W}^{-1,2}(\Omega)^3,\\
&&\textbf{div}\textbf{u}^2=0,
\end{eqnarray*}
with taking the test function $\varphi=\chi_{(0,t)}\phi$ ($\chi_{(0,t)}$ denotes the characteristic function of $(0,t)$ and $\phi\in\textbf{W}^{1,r}_{0,\textbf{div}}(\Omega)$) in (\ref{E2-R1}). Furthermore, the pressure $P$ has the form
\begin{eqnarray*}
P=P^1+\partial_tP^2~~with~P^1\in\textbf{L}^{r'}((0,T)\times\Omega)~and~P^2\in\textbf{L}^{\infty}(0,T;\textbf{L}^2(\Omega)).
\end{eqnarray*}
The solvability of above Stokes problems can be obtained by a similar proof in \cite{Br,Gal,W}. Since the presence of $\partial_tP^2$, we can not know if $P$ is an integrable function on $(0,T)\times\Omega$.

\section{The uniform estimates}
We take $\{\psi_j\}_{j=1}^{\infty}$ as a basis of $\textbf{W}_{0,\textbf{div}}^{1,r}(\Omega)$ such that $(\psi_i,\psi_j)=\delta_{i,j}$ for each $i,j=1,\ldots,\infty$ and $\psi_j\in\textbf{W}_{0,\textbf{div}}^{1,2r}(\Omega)$ for all $j$, and let $\Gamma^n$ denote the projection of $\textbf{L}^2_{\textbf{div}}(\Omega)$ onto the linear hull of $\{\psi_j\}_{j=1}^n$. Let $\textbf{u}^n\in\textbf{C}(0,T;\textbf{W}_{0,\textbf{div}}^{1,2r}(\Omega))$ and $\textbf{H}^n\in\textbf{C}(0,T;\textbf{W}_{0,\textbf{div}}^{1,2}(\Omega))$ of the form $\textbf{u}^n(t,x)=\sum_{j=1}^na_j^n(t)\psi_j(x)$ and
$\textbf{H}^n(t,x)=\sum_{j=1}^nc_j^n(t)\psi_j(x)$ (note that for $r\geq\frac{11}{5}$ it is always true that $\textbf{W}_{0,\textbf{div}}^{1,2r}(\Omega))\hookrightarrow\textbf{W}_{0,\textbf{div}}^{1,2}(\Omega)$) so that the condition $\textbf{div}\textbf{u}^n=0$ and $\textbf{div}\textbf{H}^n=0$ are automatically fulfilled and $(\rho^n,\textbf{u}^n,\theta^n,\textbf{H}^n)$ satisfy
\begin{eqnarray}\label{E3-1}
\int_{0}^T\langle\rho^n_t,z\rangle-(\rho^n\textbf{u}^n,\nabla z)dt=0,~~and~~\rho_*\leq\rho^n\leq\rho^*~~in~~(0,T)\times\Omega,
\end{eqnarray}
for any $z\in\textbf{L}^{q}(0,T;\textbf{W}^{1,q}(\Omega))$ with arbitrary $q\in[1,\infty)$;
\begin{eqnarray}\label{E3-1R}
\langle(\rho^n\textbf{u}^n)_t,\psi_i\rangle&-&(\rho^n\textbf{u}^n\otimes\textbf{u}^n,\nabla\psi_i)+(\textbf{S}(\rho^n,\theta^n,\textbf{D}(\textbf{u}^n)),\textbf{D}(\psi_i))\nonumber\\
&&+\int_{\Omega}((\textbf{H}^n)^T\nabla\psi_i\textbf{H}^n+\frac{1}{2}\nabla(|\textbf{H}^n|^2)\cdot\psi_i)dx=0,
\end{eqnarray}
for all $i=1,\ldots,n$ and a.a. $t\in[0,T]$;
\begin{eqnarray}\label{E3-1R1}
\int_0^T(\langle(\rho^nQ(\theta^n))_t,h\rangle&+&(\rho^nQ(\theta^n)\textbf{u}^n,\nabla h)-(\textbf{q}(\rho^n,\theta^n,\nabla\theta^n),\nabla h))dt\nonumber\\
&=&\int_0^T\left(\nu(|\nabla\times\textbf{H}^n|^2,h)+(\textbf{S}(\rho^n,\theta^n,\textbf{D}(\textbf{u}^n)),\textbf{D}(\textbf{u}^n)h)\right)dt,~~~~~
\end{eqnarray}
for all $h\in\textbf{L}^{\infty}(0,T;\textbf{W}^{1,q})$ with $q$ sufficiently large;

The magnetic field function $\textbf{H}^n$ satisfies
\begin{eqnarray}\label{E3-1R2}
\langle\textbf{H}^n_t,b\rangle+(\nabla\times\nu(\nabla\times\textbf{H}^n),b)-(\nabla\times(\textbf{u}\times\textbf{H}^n),b)=0,
\end{eqnarray}
for all $b\in\textbf{L}^{2}(0,T;\textbf{W}^{1,2}(\Omega))$;

$\theta^n\in\textbf{L}^{\infty}(0,T;\textbf{L}^2(\Omega))\cap\textbf{L}^p(0,T;\textbf{W}^{1,p})$ with $p=\min\{2,\frac{5(\alpha+2)}{\alpha+5}\}$, $\theta^n\geq\theta_*$ in $(0,T)\times\Omega$;

The initial data
\begin{eqnarray*}
\rho^n(0,\cdot)=\rho_0,~~\textbf{u}^n(0,\cdot)=\Gamma^n\textbf{u}_0,~~\theta^n(0,\cdot)=\theta_0~~and~~\textbf{H}^n(0,\cdot)=\textbf{H}_0,
\end{eqnarray*}
where $\Gamma^n\textbf{u}_0$ and $\theta^n_0$ a standard regularization of $\theta_0$, fulfill
\begin{eqnarray}\label{E2-1R1}
&&\Gamma^n\textbf{u}_0\longrightarrow\textbf{u}_0~~strongly~~in~\textbf{L}^2(\Omega),\\
\label{E2-1R2}
&&\theta_0^n\longrightarrow\theta_0~~strongly~~in~\textbf{L}^1(\Omega).
\end{eqnarray}
In fact, using Lemma 3 in section 4, (\ref{E3-1}) is equivalent to
\begin{eqnarray}\label{E3-2}
\int_{t_1}^{t_2}\langle\rho^n,z_t\rangle+(\rho^n\textbf{u}^n,\nabla z)ds=(\rho^n,z)(t_2)-(\rho^n,z)(t_1),
\end{eqnarray}
for $z$ a smooth function and a.a. $t_1,t_2:0\leq t_1\leq t_2\leq T$.

In the following, we give some uniform estimates with respect to $n\in\textbf{N}$. Taking $z=|\textbf{u}^n|^2$, $t_1=0$ and $t_2=t$ in (\ref{E3-2}) we have
\begin{eqnarray}\label{E3-3}
\int_{0}^{t}(\rho^n,|\textbf{u}^n|^2_t)+(\rho^n\textbf{u}^n,\nabla|\textbf{u}^n|^2)ds=(\rho^n,|\textbf{u}^n|^2)(t)-(\rho^n,|\Gamma^n\textbf{u}_0|^2).
\end{eqnarray}
Multiplying the $j$th equation in (\ref{E3-1R}) by $a_j$, then taking the sum over $j=1,\ldots,n$, using (\ref{E3-3}), and integrating the equality over $(0,t)$, we get
\begin{eqnarray}\label{E3-4}
\frac{1}{2}(\rho^n,|\textbf{u}^n|^2)(t)&+&\int_0^t(\textbf{S}(\rho^n,\theta^n,\textbf{D}(\textbf{u}^n)),\textbf{D}(\textbf{u}^n))ds\nonumber\\
&+&\int_0^t\int_{\Omega}((\textbf{H}^n)^T\nabla\textbf{u}^n\textbf{H}^n+\frac{1}{2}\nabla(|\textbf{H}^n|^2)\cdot\textbf{u}^n)dxds
=\frac{1}{2}(\rho_0,|\Gamma^n\textbf{u}_0|^2).~~~~~
\end{eqnarray}
We deal with (\ref{E3-1R2}) by the same process as in (\ref{E3-4}), and have
\begin{eqnarray}\label{E3-5}
\frac{1}{2}\|\textbf{H}^n\|^2_{\textbf{L}^2(\Omega)}&-&\int_0^t(\nabla\times(\textbf{u}^n\times\textbf{H}^n),\textbf{H}^n)ds\nonumber\\
&=&-\int_0^t(\nabla\times(\nu\nabla\times\textbf{H}^n),\textbf{H}^n)ds+\frac{1}{2}\|\Gamma^n\textbf{H}_0\|_{\textbf{L}^2(\Omega)}^2.~~~
\end{eqnarray}
Direct calculation shows that
\begin{eqnarray*}
&&\int_0^t(\nabla\times(\textbf{u}^n\times\textbf{H}^n),\textbf{H}^n)ds=\int_0^t\int_{\Omega}(\textbf{H}^n)^T\nabla\textbf{u}^n\textbf{H}^n+\frac{1}{2}\nabla(|\textbf{H}^n|^2)\cdot\textbf{u}^ndxds,\nonumber\\
&&\int_0^t(\nabla\times(\nu\nabla\times\textbf{H}^n),\textbf{H}^n)ds=\int_0^t\int_{\Omega}\nu|\nabla\times\textbf{H}^n|^2dxds.
\end{eqnarray*}
So by (\ref{E3-5}), we have
\begin{eqnarray}\label{E3-6}
\frac{1}{2}\|\textbf{H}^n\|^2_{\textbf{L}^2(\Omega)}+\int_0^t(\nu\|\nabla\times\textbf{H}^n\|_{\textbf{L}^2(\Omega)}^2-\int_{\Omega}((\textbf{H}^n)^T\nabla\textbf{u}^n\textbf{H}^n&-&\frac{1}{2}\nabla(|\textbf{H}^n|^2)\cdot\textbf{u}^n)dx)ds\nonumber\\
&=&\frac{1}{2}\|\Gamma^n\textbf{H}_0\|_{\textbf{L}^2(\Omega)}^2.~~~~
\end{eqnarray}
Summing up (\ref{E3-4}) and (\ref{E3-6}), we obtain
\begin{eqnarray}\label{E3-7}
\frac{1}{2}(\rho^n,|\textbf{u}^n|^2)(t)+\frac{1}{2}\|\textbf{H}^n\|^2_{\textbf{L}^2(\Omega)}&+&\int_0^t\nu\|\nabla\times\textbf{H}^n\|_{\textbf{L}^2(\Omega)}^2+(\textbf{S}(\rho^n,\theta^n,\textbf{D}(\textbf{u}^n)),\textbf{D}(\textbf{u}^n))ds\nonumber\\
&=&\frac{1}{2}(\rho_0,|\Gamma^n\textbf{u}_0|^2)+\frac{1}{2}\|\Gamma^n\textbf{H}_0\|_{\textbf{L}^2(\Omega)}^2.~~~~~
\end{eqnarray}
Then by the first assumption in (\ref{E1-1'}), we derive
\begin{eqnarray}\label{E3-8}
(\rho^n,|\textbf{u}^n|^2)(t)+\|\textbf{H}^n\|^2_{\textbf{L}^2(\Omega)}&+&\int_0^t2\nu\|\nabla\times\textbf{H}^n\|_{\textbf{L}^2(\Omega)}^2+(\textbf{S}(\rho^n,\theta^n,\textbf{D}(\textbf{u}^n)),\textbf{D}(\textbf{u}^n))ds\nonumber\\
&&+\underline{\mu}\int_0^t\|\textbf{D}(\textbf{u}^n)\|_{\textbf{L}^r(\Omega)}^rds\leq(\rho_0,|\Gamma^n\textbf{u}_0|^2)+\|\Gamma^n\textbf{H}_0\|_{\textbf{L}^2(\Omega)}^2\nonumber\\
&\leq& C(\rho^*,\|\textbf{u}_0\|_{\textbf{L}^2(\Omega)},\|\Gamma^n\textbf{H}_0\|_{\textbf{L}^2(\Omega)})\leq M,
\end{eqnarray}
where $M$ denotes a positive constant depending on the data and maximizes all the estimates.

Note that
\begin{eqnarray}\label{E3-8'}
0<\rho_*\leq\rho^n(t,x)\leq\rho^*<+\infty~~for~~a.a~(t,x)\in(0,T)\times\Omega.
\end{eqnarray}
Thus it follows from (\ref{E3-8}) that
\begin{eqnarray}\label{E3-9}
\sup_{t\in[0,T]}\|\textbf{u}^n(t)\|^2_{\textbf{L}^2(\Omega)}&+&\sup_{t\in[0,T]}\|\rho^n|\textbf{u}^n|^2(t)\|_{\textbf{L}^1(\Omega)}+\sup_{t\in[0,T]}\|\textbf{H}^n(t)\|^2_{\textbf{L}^2(\Omega)}\leq M.~~~~~~~~
\end{eqnarray}
By Korn's inequality, (\ref{E1-1'}) and (\ref{E3-9}), we have
\begin{eqnarray}\label{E3-10}
0<\|\textbf{H}^n\|^2_{\textbf{L}^2(\Omega)}+\int_0^T2\nu\|\nabla\times\textbf{H}^n\|_{\textbf{L}^2(\Omega)}^2dt &+&\int_0^T(\textbf{S}(\rho^n,\theta^n,\textbf{D}(\textbf{u}^n)),\textbf{D}(\textbf{u}^n))dt\nonumber\\
&&+\underline{\mu}\int_0^T\|\nabla\textbf{u}^n\|_{\textbf{L}^r(\Omega)}^rds\leq M,~~~~~~~
\end{eqnarray}
\begin{eqnarray}\label{E3-11}
\int_0^T\|\textbf{S}(\rho^n,\theta^n,\textbf{D}(\textbf{u}^n))\|^{r'}_{\textbf{L}^{r'}(\Omega)}ds\leq M.
\end{eqnarray}
Using Gagliardo-Nirenberg interpolation inequality, the inequality (\ref{E3-10}) implies that
\begin{eqnarray}\label{E3-10'}
\int_0^T\|\textbf{u}^n\|_{\textbf{L}^{\frac{5r}{3}}}^{\frac{5r}{3}}dt\leq M,~~\int_0^T\|\rho^n\textbf{u}^n\|_{\textbf{L}^{\frac{5r}{3}}}^{\frac{5r}{3}}dt\leq M
\end{eqnarray}
and
\begin{eqnarray}\label{E3-10'R}
\textbf{H}^n\in\textbf{L}^{2}(0,T;\textbf{H}^1(\Omega)).
\end{eqnarray}
Using (\ref{E3-8'}) and (\ref{E3-10}), by H\"{o}lder inequality and the fact $\textbf{W}^{1,p}(\Omega)\hookrightarrow\textbf{L}^q(\Omega)$ with $p\leq q\leq\infty$, we derive
\begin{eqnarray}\label{E3-12}
\int_0^T|(\rho^n\textbf{u}^n\otimes\textbf{u}^n,\nabla\textbf{u}^n)|dt\leq
\leq M,
\end{eqnarray}
where we require that
\begin{eqnarray}\label{E3-12'}
r\geq\frac{11}{5},
\end{eqnarray}
which gives one of the restriction for $r$ in our main result.

It follows from H\"{o}lder inequality that
\begin{eqnarray*}
&&\|(\textbf{H}^n)^T\nabla\textbf{u}^n\textbf{H}^n\|_{\textbf{L}^1(\Omega)}\leq C(\|\textbf{H}^n\|^2_{\textbf{L}^2(\Omega)}+\|\textbf{u}^n\|_{\textbf{H}^1(\Omega)})\leq M,\nonumber\\
&&\|\nabla(|\textbf{H}^n|^2)\cdot\textbf{u}^n\|_{\textbf{L}^1(\Omega)}\leq
C(\|\textbf{H}^n\|^2_{\textbf{H}^1(\Omega)}+\|\textbf{u}^n\|_{\textbf{L}^2(\Omega)})\leq M.
\end{eqnarray*}
Combining above estimates, for $1<p_1\leq r$, we deduce from (\ref{E3-1R})-(\ref{E3-1R1}) to
\begin{eqnarray}\label{E3-13'}
\int_0^T\|\rho^n_t\|^{p_1}_{(\textbf{W}^{1,\frac{p_1}{p_1-1}}(\Omega))^*}dt\leq M,~~\int_0^T\|\partial_t(\rho^n\textbf{u}^n)\|^{r'}_{(\textbf{W}^{-1,r'}_{\textbf{div}}(\Omega))}dt\leq M.
\end{eqnarray}
Let $h=1$ in (\ref{E3-1R1}). Using (\ref{E3-8'}) and (\ref{E3-10}), we have
\begin{eqnarray*}
\sup_{t\in[0,T]}\|\rho^nQ(\theta^n)\|_{\textbf{L}^1(\Omega)}+\sup_{t\in[0,T]}\|Q(\theta^n)\|_{\textbf{L}^1(\Omega)}\leq M.
\end{eqnarray*}
By (\ref{E2-1R}) and above estimate,
\begin{eqnarray}\label{E3-13}
\sup_{t\in[0,T]}\|\rho^n\theta^n\|_{\textbf{L}^1(\Omega)}+\sup_{t\in[0,T]}\|\theta^n\|_{\textbf{L}^1(\Omega)}\leq M.
\end{eqnarray}
Now we turn to estimate the temperature. Note that
\begin{eqnarray*}
\theta^n(t,x)\geq\theta^*~~for~~a.a.~(t,x)\in(0,T)\times\Omega.
\end{eqnarray*}
Take $h=-(\theta^n)^{-\lambda}$ with $0<\lambda<1$ in (\ref{E3-1R1}). Then by (\ref{E1-2'}) and (\ref{E3-13}),
\begin{eqnarray}\label{E3-14}
&&\|(\theta^n)^{-\lambda}\|_{\textbf{L}^{\infty}((0,T)\times\Omega)}\leq M,\\
\label{E3-14'}
&&\int_0^T\|(\theta^n)^{\frac{\alpha-\lambda-1}{2}}\nabla\theta^n\|_{\textbf{L}^2(\Omega)}^2dt\leq M.
\end{eqnarray}
By a contradiction argument, we can easily get
\begin{eqnarray*}
&&\|(\theta^n)^{\frac{\alpha-\lambda+1}{2}}\|_{\textbf{L}^2(\Omega)}\leq C(\|\theta^n\|^{\frac{\alpha-\lambda+1}{2}}_{\textbf{L}^1(\Omega)}+\|\nabla(\theta^n)^{\frac{\alpha-\lambda+1}{2}}\|_{\textbf{L}^1(\Omega)})~~~~if~\frac{\alpha-\lambda+1}{2}>0,\\
&&\|(\theta^n)^{\frac{\alpha-\lambda+1}{2}}\|_{\textbf{L}^2(\Omega)}\leq C~~~~if~\frac{\alpha-\lambda+1}{2}\leq0.
\end{eqnarray*}
Thus by these estimates and (\ref{E3-14}), we have
\begin{eqnarray*}
\int_0^T\|(\theta^n)^{\frac{\alpha-\lambda-1}{2}}\|_{\textbf{W}^{1,2}(\Omega)}^2dt\leq M.
\end{eqnarray*}
Furthermore,  by $\textbf{W}^{1,2}(\Omega)\hookrightarrow\textbf{L}^2(\Omega)$, it holds
\begin{eqnarray}\label{E3-15}
\int_0^T\|\theta^n\|^{\alpha-\lambda-1}_{\textbf{L}^{3(\alpha-\lambda-1)}(\Omega)}dt\leq M.
\end{eqnarray}
By the standard interpolation of (\ref{E3-15}) with (\ref{E3-13}), for $\alpha>-\frac{2}{3}$, we derive
\begin{eqnarray}\label{E3-15'}
\int_0^T\|\theta^n\|_{\textbf{L}^s(\Omega)}^sdt\leq M~~for~~all~s\in[1,\frac{5}{3}+\alpha).
\end{eqnarray}
Using (\ref{E1-3'}), (\ref{E3-14'}) and (\ref{E3-15'}), for $1\leq m<\frac{5+3\alpha}{4+3\alpha}$, we have
\begin{eqnarray}\label{E3-15''}
&&\int_{(0,T)\times\Omega}|\kappa_0(\rho^n,\theta^n)\nabla\theta^n|^mdxdt\nonumber\\
&\leq&\underline{\kappa}^m\int_{(0,T)\times\Omega}|\theta^n|^{m\alpha}|\nabla\theta^n|^mdxdt\nonumber\\
&=&\underline{\kappa}\int_{(0,T)\times\Omega}(\theta^n)^{\frac{m(\alpha-\lambda-1)}{2}}|\nabla\theta^n|^m(\theta^n)^{\frac{m(\alpha+\lambda+1)}{2}}dxdt\nonumber\\
&\leq&\underline{\kappa}\|(\theta^n)^{\frac{\alpha-\lambda-1}{2}}|\nabla\theta^n|\|_{\textbf{L}^2((0,T)\times\Omega)}
\|(\theta^n)^{\frac{\alpha+\lambda+1}{2}}\|_{\textbf{L}^{\frac{2}{\alpha+\lambda+1}(\frac{5+3\alpha}{3}-\delta)}((0,T)\times\Omega)}\nonumber\\
&\leq&M.
\end{eqnarray}
On the other hand, by (\ref{E3-8'}) and (\ref{E3-13}), using the interpolation inequality, H\"{o}lder's inequality and the standard Sobolev imbedding, we have that for some $\beta>1$
\begin{eqnarray}\label{E3-16}
\int_0^T\|\rho^n\textbf{u}^nQ(\theta^n)\|^{\beta}_{\beta}dt&\leq&\int_0^T\rho^*\|\textbf{u}^n\|^{\beta}_{\textbf{L}^{p_2}(\Omega)}\|\theta^n\|_{\textbf{L}^{p_3}(\Omega)}^{\beta}dt\nonumber\\
&\leq&\int_0^T\rho^*\|\textbf{u}^n\|^{\beta}_{\textbf{W}^{1,r}(\Omega)}\|\theta^n\|_{\textbf{L}^{1}(\Omega)}^{(1-\upsilon)\beta}\|\theta^n\|_{\textbf{L}^{p_4}(\Omega)}^{\beta\upsilon}dt\nonumber\\
&\leq&C\int_0^T\|\textbf{u}^n\|^{\beta}_{\textbf{W}^{1,r}(\Omega)}\|\theta^n\|_{\textbf{L}^{p_4}(\Omega)}^{\beta\upsilon}dt\nonumber\\
&\leq&C\left(\int_0^T\|\textbf{u}^n\|^{\frac{p_4\beta}{p_4-\beta\upsilon}}_{\textbf{W}^{1,r}(\Omega)}dt\right)^{\frac{p_4-\beta\upsilon}{p_4}}\left(\int_0^T\|\theta^n\|_{\textbf{L}^{p_4}(\Omega)}^{p_4}dt\right)^{\frac{\beta\upsilon}{p_4}},~~~~~~
\end{eqnarray}
where
\begin{eqnarray*}
&&\frac{1}{\beta}=\frac{1}{p_2}+\frac{1}{p_3},\nonumber\\
&&\frac{\beta}{p_3}=\frac{1}{\varpi}+\frac{\beta-\frac{1}{\varpi}}{p_4},~~\frac{1}{\varpi}=(1-\upsilon)\beta.
\end{eqnarray*}
Direct computation shows that
\begin{eqnarray*}
p_2=\frac{r\varpi\beta^2-r\beta}{\varpi\beta^2-\beta(r+1)+\frac{r}{\varpi}},~~p_3=\frac{p_2\beta}{p_2-\beta},~~p_4=\frac{p_2(\varpi\beta-1)}{(\varpi-1)p_2-\varpi\beta}.
\end{eqnarray*}
Note that $1\leq p_4<\frac{5}{3}+\alpha$ and $\alpha>-\frac{2}{3}$. We get from (\ref{E3-16}) that
\begin{eqnarray}\label{E3-17}
\int_0^T\|\rho^n\textbf{u}^n\theta^n\|^{\beta}_{\beta}dt\leq
C\left(\int_0^T\|\textbf{u}^n\|^{\frac{p_4\beta}{p_4-\beta+\frac{1}{\varpi}}}_{\textbf{W}^{1,r}(\Omega)}dt\right)^{\frac{p_4-\beta+\frac{1}{\varpi}}{p_4}}
\left(\int_0^T\|\theta^n\|_{\textbf{L}^{p_4}(\Omega)}^{p_4}dt\right)^{\frac{\beta-\frac{1}{\varpi}}{p_4}},~~~~
\end{eqnarray}
where
\begin{eqnarray}\label{E3-18}
p_2\geq r,~~1<\beta<p_3\leq\varpi\beta,~~\beta<\frac{r}{2},~~1<\varpi<\frac{r}{\beta}.
\end{eqnarray}
Hence we require the restriction that $r\geq\frac{11}{5}$.

Finally, by (\ref{E3-1R1}), (\ref{E3-10}) and (\ref{E3-15''}), for sufficiently large $q$, we deduce that
\begin{eqnarray}\label{E3-18'}
\|\partial_t(\rho^n\theta^n)\|_{\textbf{L}^1(0,T;(\textbf{W}^{1,q}(\Omega))^*)}&=&\int_0^T\|\partial_t(\rho^n\theta^n)(s)\|_{(\textbf{W}^{1,q}(\Omega))^*}ds\nonumber\\
&=&\int_0^T\sup_{\|h\|_{\textbf{W}^{1,q}(\Omega)}\leq1}|\langle (\rho^n\theta^n)_t(s),h\rangle|ds\nonumber\\
&\leq&\|g^n(s)\|_{\textbf{L}^1(\Omega)}\leq M.
\end{eqnarray}
It follows from (\ref{E3-10'}) that
\begin{eqnarray*}
\|\textbf{u}^n\times\textbf{H}^n\|_{\textbf{L}^{\frac{2r}{r+2}}(\Omega)}\leq M,
\end{eqnarray*}
which implies that
\begin{eqnarray}\label{E8-2}
\nabla\times\textbf{H}^n\in\textbf{L}^2(0,T;\textbf{L}^{2}(\Omega)),
\end{eqnarray}
\begin{eqnarray}\label{E8-5}
\textbf{u}^n\times\textbf{H}^n\in\textbf{L}^r(0,T;\textbf{L}^{\frac{2r}{r+2}}(\Omega)).
\end{eqnarray}
Note that $r\geq\frac{11}{5}$. By (\ref{E8-2}), (\ref{E8-5}) and (\ref{E3-1R2}), we drive
\begin{eqnarray}\label{E3-18R}
\|\partial_t\textbf{H}^n\|_{\textbf{L}^2(0,T;(\textbf{W}^{1,2}(\Omega))^*)}\leq M.
\end{eqnarray}
Using Alaoglu-Bourbaki theorem and the uniform estimates (\ref{E3-8'}), (\ref{E3-9})-(\ref{E3-10'}), (\ref{E3-13})-(\ref{E3-15'}), (\ref{E3-18'})-(\ref{E3-18R}), for $n\in\textbf{N}$, there exist a subsequences (denoted by itself) $\{\rho^n\}$, $\{\textbf{u}^n\}$, $\{\theta^n\}$ and $\{\textbf{H}^n\}$ , and $(\rho,\textbf{u},\theta,\textbf{H})$ such that
\begin{eqnarray}\label{E3-19}
\rho^n\rightharpoonup\rho&&weakly~~in~~\textbf{L}^p((0,T)\times\Omega)~for~any~q\in[1,\infty)\nonumber\\
&&and~*-weakly~in~\textbf{L}^{\infty}((0,T)\times\Omega),
\end{eqnarray}
\begin{eqnarray}\label{E3-19'}
0<\rho_*\leq\rho(t,x)\leq\rho^*<\infty~~for~~a.a.~(t,x)\in(0,T)\times\Omega,
\end{eqnarray}
\begin{eqnarray}\label{E3-19''}
\rho_t^n\rightharpoonup\rho_t~~weakly~~in~\textbf{L}^{\frac{5r}{3}}(0,T;(\textbf{W}^{1,\frac{5r}{5r-3}}(\Omega))^*),
\end{eqnarray}
\begin{eqnarray}\label{E3-19R1}
&&\textbf{u}^n\rightharpoonup\textbf{u}~~weakly~~in~~\textbf{L}^r(0,T;\textbf{W}_{0,\textbf{div}}^{1,r}(\Omega))~~and~~\textbf{L}^{\frac{5r}{3}}((0,T)\times\Omega)^3,\nonumber\\
&&~~*-weakly~~in~~\textbf{L}^{\infty}(0,T;\textbf{L}^2_{\textbf{div}}(\Omega)),
\end{eqnarray}
\begin{eqnarray*}
\theta^n\rightharpoonup\theta~~weakly~~in~~\textbf{L}^q((0,T)\times\Omega)~~for~~any~~q\in[1,\frac{5+3\alpha}{3}),
\end{eqnarray*}
\begin{eqnarray}\label{E3-19R2}
\theta^n\geq\theta^*>0~~for~~a.a.~(t,x)\in(0,T)\times\Omega,
\end{eqnarray}
\begin{eqnarray}\label{E3-22}
\textbf{H}^n\rightharpoonup\textbf{H}&&weakly~~in~~\textbf{L}^2(0,T;\textbf{W}_{0,\textbf{div}}^{1,2}(\Omega))\nonumber\\ &&and~~*-weakly~~in~~\textbf{L}^{\infty}(0,T;\textbf{L}^2_{\textbf{div}}(\Omega)),
\end{eqnarray}
\begin{eqnarray}\label{E8-6}
\nabla\times\textbf{H}^n\rightharpoonup\nabla\times\textbf{H}&&weakly~~in~~\textbf{L}^2(0,T;\textbf{L}^{2}(\Omega)),
\end{eqnarray}
\begin{eqnarray}\label{E8-1}
\textbf{H}^n_t\rightharpoonup\textbf{H}_t&&weakly~~in~~\textbf{L}^2(0,T;(\textbf{W}^{1,2}(\Omega))^*).
\end{eqnarray}
Furthermore, there holds
\begin{eqnarray*}
\rho^n\textbf{u}^n\rightharpoonup\overline{\rho\textbf{u}}~~weakly~~in~~\textbf{L}^{p_1}((0,T)\times\Omega)^3,
\end{eqnarray*}
\begin{eqnarray}\label{E3-22R}
\textbf{S}(\rho^n,\theta^n,\textbf{D}(\textbf{u}^n))\rightharpoonup\overline{\textbf{S}}&&weakly~~in~~\textbf{L}^{r'}((0,T)\times\Omega)^{3\times3},\nonumber\\
&&\overline{\textbf{S}}=\overline{\textbf{S}}^T,~~\textbf{S}~~being~traceless,
\end{eqnarray}
\begin{eqnarray}\label{E3-21}
(\theta^n)^{\gamma}\rightharpoonup\overline{(\theta)^{\gamma}}~~weakly~~in~~\textbf{L}^{2}((0,T);\textbf{W}^{1,2}(\Omega))~~for~\gamma\in(0,\frac{\alpha+1}{2}),
\end{eqnarray}
where
$\overline{\rho\textbf{u}}\in\textbf{L}^{p_1}(\textbf(0,T)\times\Omega)^3$, $\overline{\textbf{S}}\in\textbf{L}^{r'}(\textbf(0,T)\times\Omega)^{3\times3}$ and $\overline{\theta^{\gamma}}\in\textbf{L}^2((0,T);\textbf{W}^{1,2}(\Omega))$.

\section{The strong convergence of $\{\rho^n\}$, $\{\textbf{u}^n\}$, $\{\theta^n\}$, $\{\textbf{H}^n\}$ and $\{\textbf{D}(\textbf{u}^n)\}$}
Before we prove the strong convergence of $\{\rho^n\}$, $\{\textbf{u}^n\}$, $\{\theta^n\}$, $\{\textbf{H}^n\}$ and $\{\textbf{D}(\textbf{u}^n)\}$, we recall two main tools: the Aubin-Lions Lemma (see \cite{Simon} or \cite{Hu0}) and the Div-Curl Lemma (see \cite{Fei1,Freh2}).
\begin{lemma}(Lions-Aubin lemma)
Let $T>0$, $p\in(1,\infty)$ and $\{u_n\}_{n=1}^{\infty}$ be a bounded sequence of functions in $\textbf{L}^p(0,T;X)$ where $X$ is a Banach space. If $\{u_n\}_{n=1}^{\infty}$ is also bounded in $\textbf{L}^p(0,T;Y)$, where $Y$ is compactly imbedded in $X$ and $\{\partial_tu_n\}_{n=1}^{\infty}$ is bounded in $\textbf{L}^p(0,T;Z)$ uniformly where $X\subset Z$. Then $\{u_n\}_{n=1}^{\infty}$ is relatively compact in $\{u_n\}_{n=1}^{\infty}$.
\end{lemma}
For $\textbf{a}=(a_0,a_1,a_2,a_3)$ we set
\begin{eqnarray*}
\textbf{Div}_{t,x}\textbf{a}=(a_0)_t+\sum_{i=1}^3(a_i)_{x_i}~~and~~\textbf{Curl}_{t,x}\textbf{a}=\nabla_{t,x}\textbf{a}-(\nabla_{t,x}\textbf{a})^T,
\end{eqnarray*}
where $\nabla_{t,x}\textbf{a}=(\textbf{a}_t^T,\textbf{a}_{x_1}^T,\textbf{a}_{x_2}^T,\textbf{a}_{x_3}^T)$.
\begin{lemma}(Div-Curl Lemma)
Let $p$, $q$, $l$, $s\in(1,\infty)$ be such that $\frac{1}{p}+\frac{1}{q}=\frac{1}{l}$. Assume that $\{\textbf{a}^n\}$ and $\{\textbf{b}^n\}$ satisfy
\begin{eqnarray*}
&&\textbf{a}^n\rightharpoonup\textbf{a}~~weakly~~in~~\textbf{L}^p((0,T)\times\Omega)^4,\\
&&\textbf{b}^n\rightharpoonup\textbf{b}~~weakly~~in~~\textbf{L}^q((0,T)\times\Omega)^4,
\end{eqnarray*}
and $\textbf{Div}_{t,x}\textbf{a}^n$ and $\textbf{Curl}_{t,x}\textbf{b}^n$ are precompact in $\textbf{W}^{-1,s}((0,T)\times\Omega)$ and $\textbf{W}^{-1,s}((0,T)\times\Omega)^{4\times4}$, respectively.

Then
\begin{eqnarray*}
\textbf{a}^n\cdot\textbf{b}^n\rightharpoonup\textbf{a}\cdot\textbf{b}~~weakly~~in~~\textbf{L}^l((0,T)\times\Omega),
\end{eqnarray*}
where $\cdot$ represents the scalar product in $\textbf{R}^4$.
\end{lemma}
By (\ref{E3-19})-(\ref{E3-19'}) and Lions-Aubin Lemma 1, we have
\begin{eqnarray*}
\rho^n\longrightarrow\rho~~strongly~~in~~\textbf{C}([0,T];(\textbf{W}^{1,\frac{p_1}{p_1-1}})^*).
\end{eqnarray*}
We take $\textbf{a}^n=(\rho^n,\rho^nu^n_1,\rho^nu_2^n,\rho^nu_3^n)$ and $\textbf{b}^n=(u_i,0,0,0)$, $i=\{1,2,3\}$.
It follows from (\ref{E3-19})-(\ref{E3-19'}) that
\begin{eqnarray*}
\textbf{a}^n\rightharpoonup(\rho,\overline{\rho u_1},\overline{\rho u_2},\overline{\rho u_3})~~weakly~~in~~\textbf{L}^{q}((0,T)\times\Omega)^4~~\forall~q\in[1,p_1],
\end{eqnarray*}
\begin{eqnarray*}
\textbf{b}^n\rightharpoonup(u_i,0,0,0)~~weakly~~in~~\textbf{L}^{p_1}((0,T)\times\Omega)^4.
\end{eqnarray*}
Furthermore, we notice that $\textbf{L}^{r}((0,T)\times\Omega)^{3\times3}\hookrightarrow\hookrightarrow\textbf{W}^{-1,r}((0,T)\times\Omega)$ and
\begin{eqnarray*}
\textbf{Div}_{t,x}\textbf{a}^n=(\rho^n)_t+\textbf{div}(\rho^n\textbf{u}^n)=0,
\end{eqnarray*}
\begin{eqnarray*}
\textbf{Curl}_{t,x}\textbf{b}^n=\left(
\begin{array}{ccc}
 0& \nabla\textbf{u}^n  \\
-(\nabla\textbf{u}^n)^T &o
\end{array}
\right), ~~~~(o~denotes~zero~3\times3~matrix)
\end{eqnarray*}
and $\nabla\textbf{u}^n$ is bounded in $\textbf{L}^r((0,T)\times\Omega)^{3\times3}$. By Div-Curl Lemma 2, we have
\begin{eqnarray}\label{E3-20}
\rho^nu_i^n\rightharpoonup\rho u_i~~weakly~~in~~\textbf{L}^q((0,T)\times\Omega)~~\forall q\in[1,\frac{p_1}{2}].
\end{eqnarray}
It follows from (\ref{E3-19''}) and (\ref{E3-20}) that $\rho$ and $\textbf{u}$ satisfy (\ref{E2-2}). Take the test function of the form $\chi_{(t_1,t_2)}h$, $h\in\textbf{W}^{1,\frac{5r}{5r-3}}(\Omega)$ in (\ref{E2-2}). Then partial integration with respect to time and the density of $\textbf{L}^1((0,T);\textbf{W}^{1,\frac{5r}{5r-3}}(\Omega))$ implies that $\rho\in\textbf{C}([0,T];\textbf{L}_{weak}^{\infty}(\Omega))$, i.e.,
\begin{eqnarray}\label{E4-2R}
\lim_{t\longrightarrow t_0}(\rho(t),h)=(\rho(t_0),h)~~\forall t_0\in[0,T].
\end{eqnarray}
Using the concept of renormalized solutions to the equation (\ref{E3-1}) and following the method in \cite{Lions1}, we have
\begin{eqnarray}\label{E4-1}
\rho^n\longrightarrow\rho~~strongly~~in~~\textbf{C}([0,T];\textbf{L}^{q}(\Omega))~~\forall q\in[1,\infty)~~and~~a.e.~in~(0,t)\times\Omega,~~~~
\end{eqnarray}
and
\begin{eqnarray}\label{E4-1R}
\lim_{t\longrightarrow 0^+}\|\rho(t)-\rho_0\|_{\textbf{L}^q(\Omega)}=0~~\forall q\in[1,\infty).
\end{eqnarray}
In the following, we prove the convergence of $\{\textbf{u}^n\}$. By (\ref{E3-13'}) and (\ref{E3-20}), we have
\begin{eqnarray}\label{E4-2}
\partial_t(\rho^n\textbf{u}^n)\rightharpoonup\partial_t(\rho\textbf{u})~~weakly~~in~~\textbf{L}^{r'}(0,T;\textbf{W}_{\textbf{div}}^{-1,r'}(\Omega)).
\end{eqnarray}
Furthermore, it follows from (\ref{E3-9}), (\ref{E3-19R1}) and (\ref{E4-1}) that
\begin{eqnarray}\label{E4-3}
\sqrt{\rho^n}\textbf{u}^n\rightharpoonup\sqrt{\rho}\textbf{u}~~weakly~~in~~\textbf{L}^{2}((0,T)\times\Omega)^3.
\end{eqnarray}
Introduce the Helmholtz decomposition
\begin{eqnarray*}
&&\textbf{u}=\mathcal{H}[\textbf{u}]+\mathcal{H}^{\perp}[\textbf{u}],\\
&&\mathcal{H}^{\perp}[\textbf{u}]=\nabla\phi,~~\mathcal{H}[\textbf{u}]=\textbf{curl}\varphi,
\end{eqnarray*}
where $\phi$ is given by the solution to the Neumann problem
\begin{eqnarray*}
&&\triangle\phi=\textbf{div}\textbf{u}~~x\in\Omega,\\
&&\frac{\partial\phi}{\partial\textbf{n}}=0~~x\in\partial\Omega,~~\int_{\Omega}\phi dx=0,
\end{eqnarray*}
and $\varphi$ satisfies the following elliptic problem
\begin{eqnarray*}
&&\textbf{curl}\mathcal{H}^{\perp}[\textbf{u}]=\nabla\textbf{u}=\omega,~~x\in\Omega,\\
&&\textbf{div}\mathcal{H}^{\perp}[\textbf{u}]=0,~~x\in\Omega,\\
&&\mathcal{H}^{\perp}[\textbf{u}]\cdot\textbf{n}=0~~x\in\partial\Omega.
\end{eqnarray*}
Then by the Aubin-Lions Lemma 1, (\ref{E3-13'}), (\ref{E3-20}), (\ref{E4-3}) and (\ref{E4-2}), we have
\begin{eqnarray}\label{E4-4}
(\rho^n\textbf{u}^n)_{\textbf{div}}\longrightarrow(\rho\textbf{u})_{\textbf{div}}~~strongly~~in~~\textbf{C}(0,T;\textbf{W}_{\textbf{div}}^{-1,r'}(\Omega)),
\end{eqnarray}
which together with (\ref{E3-19R1}) implies
\begin{eqnarray}\label{E4-4R1}
\sqrt{\rho^n}\textbf{u}^n(t)\longrightarrow\sqrt{\rho}\textbf{u}(t)~~strongly~~in~~\textbf{L}^{2}((0,T)\times\Omega)^3~~for~almost~all~t\in[0,T].~~~~
\end{eqnarray}
This together with (\ref{E3-13'}) and (\ref{E4-1}), (\ref{E3-19'})-(\ref{E3-19R1}) shows that
\begin{eqnarray}\label{E4-1'}
\textbf{u}^n\longrightarrow\textbf{u}~~strongly~~in~~\textbf{L}^{q}((0,T)\times\Omega)^3~~for~all~q\in[1,p_1)~~and~~a.e.~in~(0,T)\times\Omega,~~~~
\end{eqnarray}
and for $r\geq\frac{11}{5}$
\begin{eqnarray}\label{E4-4R}
\rho^n\textbf{u}^n\otimes\textbf{u}^n\rightharpoonup\rho\textbf{u}\otimes\textbf{u}~~weakly~~in~~\textbf{L}^{r'}(0,T;\textbf{L}^{r'}(\Omega)).
\end{eqnarray}
Now we give the strong convergence properties of $\{\theta^n\}$. We take $\textbf{a}^n=(\rho^nQ(\theta^n),R_1^n,R_2^n,R_3^n)$ with $R^n:=\rho^n\theta^n\textbf{u}^n+\kappa(\rho^n,\theta^n)\nabla\theta^n$, and $\textbf{b}^n=((\theta^n)^{\gamma},0,0,0)$ with $\beta\in(0,\frac{\alpha+1}{2})$ rather small. Then by (\ref{E3-15''}), (\ref{E3-17}), (\ref{E3-19R2}) and (\ref{E4-1}), we observe that for $\frac{1}{a}+\frac{1}{b}=1$,
\begin{eqnarray*}
&&\textbf{a}^n\rightharpoonup\textbf{a}~~weakly~~in~~\textbf{L}^a((0,T)\times\Omega),\\
&&\textbf{b}^n\rightharpoonup(\overline{\theta^{\beta}},0,0,0)~~weakly~~in~~\textbf{L}^b((0,T)\times\Omega).
\end{eqnarray*}
For $1<s_1<2$, we have
\begin{eqnarray*}
\textbf{Div}_{t,x}\textbf{a}^n&=&\partial_t(\rho^nQ(\theta^n))+\textbf{div}R^n=\nu|\nabla\times\textbf{H}^n|^2\nonumber\\
&&+S(\rho^n,\theta^n,\textbf{D}(\textbf{u}^n))\cdot\textbf{D}(\textbf{u}^n) \subset\textbf{L}^1((0,T)\times\Omega)\hookrightarrow\hookrightarrow\textbf{W}^{-1,s_1}((0,T)\times\Omega),
\end{eqnarray*}
and
\begin{eqnarray*}
\textbf{Curl}_{t,x}\textbf{b}^n=\left(
\begin{array}{ccc}
 0& \nabla(\theta^n)^{\gamma}  \\
-(\nabla(\theta^n)^{\gamma})^T &o,
\end{array}
\right)\subset\textbf{L}^2((0,T)\times\Omega)^{4\times4}\hookrightarrow\hookrightarrow\textbf{W}^{-1,2}((0,T)\times\Omega)^{4\times4}.
\end{eqnarray*}
Using the Div-Curl Lemma 2, for some $\sigma_1,~~\sigma_2>0$, we have
\begin{eqnarray}\label{E4-5}
&&\rho^n(\theta^n)^{1+\gamma}\rightharpoonup\rho\theta\overline{\theta^{\gamma}}~~weakly~~in~~\textbf{L}^{1+\sigma_1}((0,T)\times\Omega),\nonumber\\
&&\rho(\theta^n)^{1+\gamma}\rightharpoonup\rho\theta\overline{\theta^{\gamma}}~~weakly~~in~~\textbf{L}^{1+\sigma_2}((0,T)\times\Omega).
\end{eqnarray}
Then by Minty's method and (\ref{E4-5}), we get
\begin{eqnarray}\label{E4-6}
\overline{\theta^{\gamma}}=\theta^{\gamma}~~a.e.~~in~~(0,T)\times\Omega.
\end{eqnarray}
It follows from (\ref{E4-5})-(\ref{E4-6}) that
\begin{eqnarray*}
\rho^{\frac{1}{\gamma+1}}\theta^n\longrightarrow\rho^{\frac{1}{\gamma+1}}\theta~~strongly~~in~~\textbf{L}^{\gamma+1}((0,T)\times\Omega),
\end{eqnarray*}
which together with (\ref{E3-15'}), (\ref{E3-19'}) and (\ref{E3-19R2}) implies
\begin{eqnarray}\label{E4-7}
\theta^n\longrightarrow\theta~~strongly~~in~~\textbf{L}^q((0,T)\times\Omega)~~\forall~q\in[1,\frac{5}{3}+\alpha).
\end{eqnarray}
By (\ref{E3-17}), (\ref{E3-21}), (\ref{E4-1}), (\ref{E4-1'}) and (\ref{E4-7}), we have
\begin{eqnarray}\label{E4-7R1}
\rho^n\theta^n\longrightarrow\rho\theta~~strongly~~in~~\textbf{L}^q((0,T)\times\Omega)~~\forall q\in[1,\frac{5}{3}+\alpha),
\end{eqnarray}
\begin{eqnarray*}
\rho^n\theta^n\textbf{u}^n\longrightarrow\rho\theta\textbf{u}~~strongly~~in~~\textbf{L}^1((0,T)\times\Omega),
\end{eqnarray*}
\begin{eqnarray*}
(\theta^n)^{\gamma}\rightharpoonup\theta^{\gamma}~~weakly~~in~~\textbf{L}^{2}((0,T);\textbf{W}^{1,2}(\Omega))~~for~\gamma\in(0,\frac{\alpha+1}{2}),
\end{eqnarray*}
which implies that
\begin{eqnarray}\label{E4-8R1}
\nabla(\theta^n)^{\frac{\alpha-\lambda+1}{2}}\rightharpoonup\nabla\theta^{\frac{\alpha-\lambda+1}{2}}~~weakly~~in~~\textbf{L}^{2}((0,T);\textbf{L}^{2}(\Omega)).
\end{eqnarray}
It follows from (\ref{E2-1R}) and (\ref{E4-7R1}) that
\begin{eqnarray}\label{E4-7R2}
\rho^nQ(\theta^n)\longrightarrow\rho Q(\theta)~~strongly~~in~~\textbf{L}^q((0,T)\times\Omega)~~\forall q\in[1,\frac{5}{3}+\alpha).~~
\end{eqnarray}
Note that
\begin{eqnarray}\label{E4-8R}
\textbf{q}(\rho^n,\theta^n,\nabla\theta^n)=\kappa_0(\rho^n,\theta^n)\nabla\theta^n=\frac{2}{\alpha-\lambda+1}(\theta^n)^{\frac{1+\lambda-\alpha}{2}}\kappa_0(\rho^n,\theta^n)\nabla(\theta^n)^{\frac{\alpha-\lambda+1}{2}}.~~~
\end{eqnarray}
Take $\lambda>0$ sufficiently small and $q>2$ satisfying $\frac{(\alpha+\lambda+1)q}{2}=\frac{5}{3}+\alpha-\lambda$. Then we have
\begin{eqnarray}\label{E4-8}
\int_0^T\|(\theta^n)^{\frac{1+\lambda-\alpha}{2}}\kappa_0(\rho^n,\theta^n)\|_{\textbf{L}^q(\Omega)}^qdt\leq C\int_0^T\int_{\Omega}(\theta^n)^{\frac{(\alpha+\lambda+1)q}{2}}dxds\leq M.
\end{eqnarray}
Using Vitali's theorem, (\ref{E4-1}), (\ref{E4-7}) and (\ref{E4-8}), we derive
\begin{eqnarray*}
(\theta^n)^{\frac{1+\lambda-\alpha}{2}}\kappa_0(\rho^n,\theta^n)\longrightarrow(\theta)^{\frac{1+\lambda-\alpha}{2}}\kappa_0(\rho,\theta)~~strongly~~in~~\textbf{L}^2(0,T;\textbf{L}^2(\Omega)).
\end{eqnarray*}
So by (\ref{E3-15''}), (\ref{E4-8R1})-(\ref{E4-8}), we get
\begin{eqnarray*}
\textbf{q}(\rho^n,\theta^n,\nabla\theta^n)\longrightarrow\textbf{q}(\rho,\theta,\nabla\theta)~~weakly~~in~~\textbf{L}^z((0,T)\times\Omega)~~\forall z\in(1,\frac{5+3\alpha}{4+3\alpha}).
\end{eqnarray*}
It follows from (\ref{E3-22}) and (\ref{E8-1}) that
\begin{eqnarray}\label{E4-7R}
\textbf{H}^n\longrightarrow\textbf{H}~~strongly~~in~~\textbf{L}^2((0,T)\times\Omega).
\end{eqnarray}
Furthermore, by (\ref{E3-19R1}) and (\ref{E4-7R}), we obtain
\begin{eqnarray}\label{E8-7}
\textbf{u}^n\times\textbf{H}^n\rightharpoonup\textbf{u}\times\textbf{H}&&weakly~~in~~\textbf{L}^2(0,T;\textbf{L}^{2}(\Omega)).
\end{eqnarray}
Finally, we study the convergence property of $\{\textbf{D}(\textbf{u}^n)\}$ by applying the Minty method (standard monotone operator technique).
The following integration by parts formula (the case $r\geq\frac{11}{5}$) can be obtained by a small modification of Lemma 4.1 in \cite{Freh2}. For reader's convenience, we give the proof.

We define
\begin{eqnarray*}
&&(\omega_h^+*z)(t,x):=\frac{1}{h}\int_0^hz(t+\tau,x)d\tau,~~(\omega_h^-*z)(t,x):=\frac{1}{h}\int^0_{-h}z(t+\tau,x)d\tau,\\
&&D^hz:=\frac{z(t+h,x)-z(t,x)}{h},~~D^{-h}z:=\frac{z(t,x)-z(t-h,x)}{h},
\end{eqnarray*}
where $z$ denotes any locally integrable function and $h>0$.

Then we have the following relationships:
\begin{eqnarray}\label{E8-8'}
(\omega_h^+*z)_t=D^h(z),~~(\omega_h^-*z)_t=D^{-h}(z),
\end{eqnarray}
\begin{eqnarray}\label{E8-8}
-\int_{t_0}^{t_1}f(\tau)D^hg(\tau)=\int_{t_0}^{t_1}D^{-h}f(\tau)g(\tau)d\tau&+&\frac{1}{h}\int_{t_0}^{t_0+h}f(\tau-h)g(\tau)d\tau\nonumber\\
&&-\frac{1}{h}\int_{t_1}^{t_1+h}f(\tau-h)g(\tau)d\tau.~~~
\end{eqnarray}
\begin{lemma}
Assume that $r\geq\frac{11}{5}$,
\begin{eqnarray*}
&&\rho\in\textbf{L}^{\infty}((0,T)\times\Omega)\cap\textbf{C}([0,T];\textbf{L}^q(\Omega)),~~\forall q\in[1,\infty),\\
&&\textbf{u}\in\textbf{L}^r(0,T;\textbf{W}_{0,\textbf{div}}^{1,r}(\Omega))\cap\textbf{L}^{\infty}(0,T;\textbf{L}^2(\Omega)^3),\\
&&\partial_t(\rho\textbf{u})\in(\textbf{L}^r(0,T;\textbf{W}_{0,\textbf{div}}^{1,r}(\Omega))^*,
\end{eqnarray*}
and the coupled $(\rho,\textbf{u})$ is a weak solution to $\rho_t+\textbf{div}(\rho\textbf{u})=0$, which means that for all $z\in\textbf{L}^{r}(0,T;\textbf{W}^{1,r}(\Omega))$, $z_t\in \textbf{L}^{1+\sigma}(0,T;\textbf{L}^{1+\sigma}(\Omega))$ and $\forall~t_0,t_1\in[0,T]$,
\begin{eqnarray}\label{E8-9}
\int_{t_0}^{t_1}(\rho(s),z_t(s))+(\rho(s)\textbf{u}(s),\nabla z(s))ds=(\rho(t_1),z(t_1))-(\rho(t_0),z(t_0)).
\end{eqnarray}
Then the following formula holds
\begin{eqnarray}\label{E4-10}
\int_{t_0}^{t_1}(\partial_t(\rho\textbf{u}),\textbf{u})-(\rho\textbf{u}\otimes\textbf{u},\nabla\textbf{u})dt=K(t_1)-K(t_0),~\forall~t_0,t_1\in[0,T],
\end{eqnarray}
where $K(t)$ is defined by
\begin{eqnarray*}
K(t):=\frac{1}{2}(\rho(t),|\textbf{u}(t)|^2)=\frac{1}{2}\int_{\Omega}\rho(t,x)|\textbf{u}(t,x)|^2dx=\frac{1}{2}\|(\sqrt{\rho}\textbf{u})(t)\|^2_{2}.
\end{eqnarray*}
\end{lemma}
\begin{proof}
Define
\begin{eqnarray*}
L_h:=\int_{t_0}^{t_1}\langle(\rho\textbf{u})_t,\omega_h^+*\omega_h^-*\textbf{u}\rangle dt,
\end{eqnarray*}
where $0<t_0<t_1<T$ and $\forall h\in(0,\min\{T-t_1,t_0\})$.

Integration by parts on $(t_0,t_1)$, we have
\begin{eqnarray*}
L_h=-\int_{t_0}^{t_1}(\rho\textbf{u},D^h(\omega_n^-*\textbf{u}))dt&+&((\rho\textbf{u})(t_1),(\omega_h^+*\omega_h^-*\textbf{u})(t_1))\nonumber\\
&-&((\rho\textbf{u})(t_0),(\omega_h^+*\omega_h^-*\textbf{u})(t_0))dt.
\end{eqnarray*}
Thus to prove (\ref{E4-10}), using (\ref{E8-8}), it is equivalent to show that
\begin{eqnarray}\label{E8-11}
\lim_{h\rightarrow0^+}L_h&=&\int_{t_0}^{t_1}(\rho\textbf{u},\frac{1}{2}\nabla|\textbf{u}|^2)dt+[K(t_1)-K(t_0)]\nonumber\\
&=&-\lim_{h\rightarrow0^+}\int_{t_0}^{t_1}(\rho\textbf{u},D^h(\omega_h^-*\textbf{u}))dt+2[K(t_1)-K(t_0)]\nonumber\\
&=&\lim_{h\rightarrow0^+}\int_{t_0}^{t_1}(D^{-h}(\rho\textbf{u}),\omega_h^-*\textbf{u})dt.
\end{eqnarray}
By (\ref{E8-8'}) and $D^{-h}\rho=-\textbf{div}(\omega^{-}_{h}*(\rho\textbf{u}))$,
we derive
\begin{eqnarray*}
&&\int_{t_0}^{t_1}(D^{-h}(\rho\textbf{u}),\omega_h^-*\textbf{u})dt\nonumber\\
&=&\int_{t_0}^{t_1}(\rho D^{-h}\textbf{u},\omega_h^-*\textbf{u})dt
+\int_{t_0}^{t_1}(D^{-h}\rho)\textbf{u}(\cdot-h),\omega_h^-*\textbf{u})dt\nonumber\\
&=&\int_{t_0}^{t_1}(\rho,\frac{1}{2}|\omega_h^-*\textbf{u}|^2_t)dt+\int_{t_0}^{t_1}(\omega_h^-*(\rho\textbf{u}),\nabla(\textbf{u}(\cdot-h)\cdot(\omega_h^-*\textbf{u}))dt.
\end{eqnarray*}
Finally, we can take $z=\frac{1}{2}|\omega_h^-*\textbf{u}|^2$ in (\ref{E8-9}). Then we conclude that
\begin{eqnarray*}
&&\int_{t_0}^{t_1}(D^{-h}(\rho\textbf{u}),\omega_h^-*\textbf{u})dt\nonumber\\
&=&(\rho(t_1),\frac{1}{2}|\omega_h^-*\textbf{u}|^2(t_1))-
(\rho(t_0),\frac{1}{2}|\omega_h^-*\textbf{u}|^2(t_0))-\int_{t_0}^{t_1}(\rho\textbf{u},\frac{1}{2}|\omega_h^-*\textbf{u}|^2_t)dt\nonumber\\
&&+\int_{t_0}^{t_1}(\omega_h^-*(\rho\textbf{u}),\nabla(\textbf{u}(\cdot-h)\cdot(\omega_h^-*\textbf{u}))dt.
\end{eqnarray*}
Thus taking limit $h\rightarrow0$ in the above inequality, and recalling (\ref{E8-11}), we can obtain (\ref{E4-10}) for almost $t_0$ and $t_1$ in $(0,T)$.
\end{proof}
Using (\ref{E3-22}), (\ref{E3-22R}), (\ref{E4-2}), (\ref{E4-4R}), (\ref{E8-6})-(\ref{E8-1}) and (\ref{E8-7}), we take the limit $n\longrightarrow\infty$ in (\ref{E3-1R}) and (\ref{E3-1R2}) and get
\begin{eqnarray}\label{E4-9}
\int_0^T\langle(\rho\textbf{u})_t,\psi\rangle&-&(\rho\textbf{u}\otimes\textbf{u},\nabla\psi)+(\overline{\textbf{S}},\textbf{D}(\psi))\nonumber\\
&&+\int_{\Omega}((\textbf{H})^T\nabla\psi\textbf{H}+\frac{1}{2}\nabla(|\textbf{H}|^2)\cdot\psi)dxdt=0,~\forall\psi\in\textbf{W}_{0,\textbf{div}}^{1,r}(\Omega),~~~~~~
\end{eqnarray}
\begin{eqnarray}\label{E4-9R}
\int_0^T\langle\textbf{H}_t,b\rangle-(\nabla\times(\textbf{u}\times\textbf{H}),b)+(\nabla\times(\nu\nabla\times\textbf{H}),b)dt=0~~\forall b\in\textbf{W}^{1,2}_{0,\textbf{div}}(\Omega).~~~~~
\end{eqnarray}
Let $\chi_{(t_0,t_1)}(s)$ be the characteristic function of the interval $(t_0,t_1)$ and $\phi\in\textbf{W}^{1,r}_{0,\textbf{div}}(\Omega)$. We take the test function $\varphi(s,x)=\chi_{(t_0,t_1)}(s)\phi(x)$ in (\ref{E4-10}), then performing partial integration with respect to time in the first term, for $\forall\phi\in\textbf{W}_{0,\textbf{div}}^{1,r}(\Omega)$, we get
\begin{eqnarray*}
(\rho(t)\textbf{u}(t),\phi)&-&(\rho_0\textbf{u}_0,\phi)\nonumber\\
&=&\int_0^t(\rho\textbf{u}\otimes\textbf{u},\nabla\phi)-(\overline{\textbf{S}},\textbf{D}(\phi))-(\textbf{H})^T\nabla\phi\textbf{H}-\frac{1}{2}\nabla(|\textbf{H}|^2)\cdot\phi dt,
\end{eqnarray*}
which implies
\begin{eqnarray}\label{E4-11}
\lim_{t\longrightarrow0^+}(\rho(t)\textbf{u}(t)-\rho_0\textbf{u}_0,\phi)=0~~\forall\phi\in\textbf{L}_{\textbf{div}}^{2}(\Omega).
\end{eqnarray}
By (\ref{E4-4R1}), letting $n\longrightarrow\infty$ in (\ref{E3-4}), we have
\begin{eqnarray}\label{E4-12}
(\rho,|\textbf{u}|^2)(t)\leq(\rho_0,|\textbf{u}_0|^2)~~for~a.a.~t\in[0,T].
\end{eqnarray}
It follows from (\ref{E4-2R}), (\ref{E4-11})-(\ref{E4-12}) that
\begin{eqnarray*}
\lim_{t\longrightarrow0^+}\|\sqrt{\rho(t)}(\textbf{u}(t)-\textbf{u}_0)\|^2_{\textbf{L}^2(\Omega)}
=\lim_{t\longrightarrow0^+}\left((\rho,|\textbf{u}|^2)(t)-2((\rho\textbf{u})(t),\textbf{u}_0)+(\rho(t),|\textbf{u}_0|^2)\right)=0.
\end{eqnarray*}
This together with (\ref{E4-1R}) implies that
\begin{eqnarray*}
\lim_{t\longrightarrow0^+}(\rho,|\textbf{u}|^2)=(\rho_0,|\textbf{u}_0|^2).
\end{eqnarray*}
Then by Lemma 3, there exists $T_{\sigma}\in(0,T]$ such that
\begin{eqnarray}\label{E4-13}
\int_0^{T_{\sigma}}\langle\partial_t(\rho\textbf{u}),\textbf{u}\rangle-(\rho\textbf{u}\otimes\textbf{u},\textbf{u})dt
=\frac{1}{2}(\rho,|\textbf{u}|^2)(T_{\sigma})-\frac{1}{2}(\rho_0,|\textbf{u}_0|^2),
\end{eqnarray}
where $T_{\sigma}\longrightarrow T$ as $\sigma\longrightarrow0$.

Taking $\psi=\chi_{(t_0,t_1)}\textbf{u}$ and $b=\chi_{(t_0,t_1)}\textbf{H}$ in (\ref{E4-9}) and (\ref{E4-9R}), by (\ref{E4-13}), we derive
\begin{eqnarray}\label{E4-14}
\frac{1}{2}(\rho,|\textbf{u}|^2)(T_{\sigma})+\int_0^{T_{\sigma}}\left((\overline{\textbf{S}},\textbf{D}(\textbf{u}))+\int_{\Omega}((\textbf{H})^T\nabla\textbf{u}\textbf{H}+\frac{1}{2}\nabla(|\textbf{H}|^2)\cdot\textbf{u})dx \right)dt=\frac{1}{2}(\rho_0,|\textbf{u}_0|^2),~~~~
\end{eqnarray}
and
\begin{eqnarray}\label{E4-14R}
\frac{1}{2}\|\textbf{H}(T_{\sigma})\|^2_{\textbf{L}^2(\Omega)}+\int_0^{T_{\sigma}}(\nu\|\nabla\times\textbf{H}\|_{\textbf{L}^2(\Omega)}^2&-&\int_{\Omega}((\textbf{H})^T\nabla\textbf{u}\textbf{H}-\frac{1}{2}\nabla(|\textbf{H}|^2)\cdot\textbf{u})dx)dt\nonumber\\
&=&\frac{1}{2}\|\textbf{H}_0\|_{\textbf{L}^2(\Omega)}^2.
\end{eqnarray}
Summing up (\ref{E4-14}) and (\ref{E4-14R}), we obtain
\begin{eqnarray}\label{E4-14R1}
\frac{1}{2}((\rho,|\textbf{u}|^2)(T_{\sigma})+\|\textbf{H}(T_{\sigma})\|^2_{\textbf{L}^2(\Omega)})&+&\int_0^{T_{\sigma}}\left((\overline{\textbf{S}},\textbf{D}(\textbf{u}))+\nu\|\nabla\times\textbf{H}\|_{\textbf{L}^2(\Omega)}^2 \right)dt\nonumber\\
&=&\frac{1}{2}((\rho_0,|\textbf{u}_0|^2)+\|\textbf{H}_0\|_{\textbf{L}^2(\Omega)}^2).
\end{eqnarray}
By (\ref{E1-1'}), for $\textbf{B}\in\textbf{L}^r((0,T)\times\Omega)^{3\times3}$ and $\textbf{B}=\textbf{B}^T$, we have
\begin{eqnarray*}
\int_0^{T_{\sigma}}(\textbf{S}(\rho^n,\theta^n,\textbf{D}(\textbf{u}^n))-\textbf{S}(\rho^n,\theta^n,\textbf{B}),\textbf{D}(\textbf{u}^n)-\textbf{B})dt\geq0.
\end{eqnarray*}
Using (\ref{E3-7}), we can rewrite the above inequality as
\begin{eqnarray}\label{E4-15}
0\leq\frac{1}{2}((\rho_0,|\Gamma^n\textbf{u}_0|^2)&-&(\rho^n,|\textbf{u}^n|^2)(T_{\sigma}))+\frac{1}{2}(\|\textbf{H}^n(T_{\sigma})\|^2_{\textbf{L}^2(\Omega)}-\|\textbf{H}^n_0\|^2_{\textbf{L}^2(\Omega)})
\nonumber\\
&&+\nu\int_0^{T_{\sigma}}\|\nabla\times\textbf{H}^n\|^2dt
-\int_0^{T_{\sigma}}(\textbf{S}(\rho^n,\theta^n,\textbf{D}(\textbf{u}^n)), \textbf{B})dt\nonumber\\
&&-\int_0^{T_{\sigma}}(\textbf{S}(\rho^n,\theta^n,\textbf{B}),\textbf{D}(\textbf{u}^n)-\textbf{B})dt.
\end{eqnarray}
Let $n\longrightarrow\infty$ in (\ref{E4-15}). By (\ref{E1-1'}), (\ref{E2-1R1})-(\ref{E2-1R2}), (\ref{E3-19R1}), (\ref{E3-22R}), (\ref{E4-1}), (\ref{E4-7}), (\ref{E4-7R}) and the Lebesgue dominated convergence theorem, we derive
\begin{eqnarray*}
\frac{1}{2}((\rho_0,|\textbf{u}_0|^2)&-&(\rho,|\textbf{u}|^2)(T_{\sigma}))+\frac{1}{2}(\|\textbf{H}(T_{\sigma})\|^2_{\textbf{L}^2(\Omega)}-\|\textbf{H}_0\|^2_{\textbf{L}^2(\Omega)})
+\nu\int_0^{T_{\sigma}}\|\nabla\times\textbf{H}\|_{\textbf{L}^2(\Omega)}^2dt\nonumber\\
&&-\int_0^{T_{\sigma}}(\overline{\textbf{S}},\textbf{B})dt-\int_0^{T_{\sigma}}(\textbf{S}(\rho,\theta,\textbf{B}),\textbf{D}(\textbf{u})-\textbf{B})dt\geq0,
\end{eqnarray*}
which together with (\ref{E4-14R1}) and letting $\sigma\longrightarrow0^+$ implies
\begin{eqnarray*}
\int_0^{T}(\overline{\textbf{S}}-\textbf{S}(\rho,\theta,\textbf{B}),\textbf{D}(\textbf{u})-\textbf{B})dt\geq0~~\forall\textbf{B}\in\textbf{L}^r((0,T)\times\Omega)^{3\times3},~~\textbf{B}=\textbf{B}^T.
\end{eqnarray*}
Taking $\textbf{B}=\textbf{D}(\textbf{u})\pm\tau\textbf{C}$ with any symmetric matrix $\textbf{C}\in\textbf{L}^r((0,T)\times\Omega)^{3\times3}$ and $\tau>0$, then using Minty's method, we have
\begin{eqnarray*}
\int_0^{T}(\overline{\textbf{S}}-\textbf{S}(\rho,\theta,\textbf{D}(\textbf{u})),\textbf{C})dt=0,
\end{eqnarray*}
which means that
\begin{eqnarray}\label{E4-16}
\overline{\textbf{S}}(t,x)=\textbf{S}(\rho,\theta,\textbf{D}(\textbf{u}))(t,x)~~for~a.a.~(t,x)\in(0,T)\times\Omega.
\end{eqnarray}
This together with (\ref{E4-7R}), (\ref{E4-9}) shows that (\ref{E2-R1}) is satisfied for $\rho$, $\textbf{u}$ and $\textbf{H}$.

Next we show that (\ref{E2-R2}) holds. It follows from (\ref{E2-1R1}), (\ref{E3-7}), (\ref{E4-4R1}), (\ref{E4-14R1}) and (\ref{E4-16}) that
\begin{eqnarray*}
&&\lim_{n\longrightarrow\infty}\int_0^T\left(\textbf{S}(\rho^n,\theta^n,\textbf{D}(\textbf{u}^n)),\textbf{D}(\textbf{u}^n))-\textbf{S}(\rho,\theta,\textbf{D}(\textbf{u})),\textbf{D}(\textbf{u}))\right)dt\nonumber\\
&=&\frac{1}{2}\lim_{n\longrightarrow\infty}\lim_{T_{\sigma}\longrightarrow T}\left((\rho,|\textbf{u}|^2)(T_{\sigma})-(\rho^n,|\textbf{u}^n|^2)(T_{\sigma})-(\rho_0,|\textbf{u}_0|^2)+(\rho,|\Gamma^n\textbf{u}_0|^2)\right)=0.
\end{eqnarray*}
By (\ref{E1-1'}), (\ref{E3-22R}), (\ref{E4-1}), (\ref{E4-1'}), (\ref{E4-7}) and (\ref{E4-7R}), using the Lebesgue dominated convergence theorem, we have
\begin{eqnarray*}
\lim_{n\longrightarrow\infty}\int_0^T\left(\textbf{S}(\rho^n,\theta^n,\textbf{D}(\textbf{u}))-\textbf{S}(\rho,\theta,\textbf{D}(\textbf{u})),\textbf{D}(\textbf{u}^n-\textbf{u})\right)dt=0,
\end{eqnarray*}
\begin{eqnarray*}
\lim_{n\longrightarrow\infty}\int_0^T\left(\textbf{S}(\rho,\theta,\textbf{D}(\textbf{u})),\textbf{D}(\textbf{u}^n-\textbf{u})\right)dt=0,
\end{eqnarray*}
\begin{eqnarray*}
\lim_{n\longrightarrow\infty}\int_0^T\left(\textbf{S}(\rho^n,\theta^n,\textbf{D}(\textbf{u}^n))-\textbf{S}(\rho,\theta,\textbf{D}(\textbf{u})),\textbf{D}(\textbf{u})\right)dt=0.
\end{eqnarray*}
Above four convergence results show that
\begin{eqnarray*}
\lim_{n\longrightarrow\infty}\int_0^T\left(\textbf{S}(\rho^n,\theta^n,\textbf{D}(\textbf{u}^n))-\textbf{S}(\rho^n,\theta^n,\textbf{D}(\textbf{u})),\textbf{D}(\textbf{u}^n-\textbf{u})\right)dt=0,
\end{eqnarray*}
which combing with (\ref{E1-1'}) shows that
\begin{eqnarray*}
0\leq(\textbf{S}(\rho^n,\theta^n,\textbf{D}(\textbf{u}^n))-\textbf{S}(\rho^n,\theta^n,\textbf{D}(\textbf{u})))\cdot\textbf{D}(\textbf{u}^n-\textbf{u})\longrightarrow0~~strongly~~in~~\textbf{L}^1((0,T)\times\Omega).
\end{eqnarray*}
Consequently, for any $h\in\textbf{L}^{\infty}((0,T)\times\Omega)$, we get
\begin{eqnarray*}
0&=&\lim_{n\longrightarrow\infty}\int_{(0,T)\times\Omega}\textbf{S}(\rho^n,\theta^n,\textbf{D}(\textbf{u}^n))-\textbf{S}(\rho^n,\theta^n,\textbf{D}(\textbf{u})))\cdot\textbf{D}(\textbf{u}^n-\textbf{u})h dxdt\nonumber\\
&=&\lim_{n\longrightarrow\infty}\int_{(0,T)\times\Omega}\textbf{S}(\rho^n,\theta^n,\textbf{D}(\textbf{u}^n))\cdot\textbf{D}(\textbf{u}^n)hdxdt\nonumber\\
&&+\lim_{n\longrightarrow\infty}\int_{(0,T)\times\Omega}\textbf{S}(\rho^n,\theta^n,\textbf{D}(\textbf{u}))\cdot\textbf{D}(\textbf{u})hdxdt\nonumber\\
&&-\lim_{n\longrightarrow\infty}\int_{(0,T)\times\Omega}\textbf{S}(\rho^n,\theta^n,\textbf{D}(\textbf{u}))\cdot\textbf{D}(\textbf{u}^n)hdxdt\nonumber\\
&&-\lim_{n\longrightarrow\infty}\int_{(0,T)\times\Omega}\textbf{S}(\rho^n,\theta^n,\textbf{D}(\textbf{u}^n))\cdot\textbf{D}(\textbf{u})hdxdt,
\end{eqnarray*}
which together with (\ref{E3-22R}), (\ref{E4-1}), (\ref{E4-1'}), (\ref{E4-7}) and (\ref{E4-7R}), the Lebesgue dominated convergence theorem implies
\begin{eqnarray*}
&&\lim_{n\longrightarrow\infty}\int_{(0,T)\times\Omega}\textbf{S}(\rho^n,\theta^n,\textbf{D}(\textbf{u}^n))\cdot\textbf{D}(\textbf{u}^n)hdxdt\nonumber\\
&&=\lim_{n\longrightarrow\infty}\int_{(0,T)\times\Omega}\textbf{S}(\rho,\theta,\textbf{D}(\textbf{u}))\cdot\textbf{D}(\textbf{u})hdxdt.
\end{eqnarray*}
Thus it follows from the above convergence properties, (\ref{E4-7R2}) and letting $n\longrightarrow\infty$ in (\ref{E3-1R1}) that (\ref{E2-R2}) holds.

\section{Appendix-Existence of solutions for the approximation system}
In this section, we establish the existence of solutions to the approximation problem (\ref{E1-1})-(\ref{E1-2}), (\ref{E1-2R}) and (\ref{E1-4}) for any fixed $n\in\textbf{N}$. We introduce a two-level approximation depending on the parameters $\epsilon>0$ and $k\in\textbf{N}$, then study the behavior of the system for $\epsilon\longrightarrow0$ and $k\longrightarrow\infty$.

Let $\{\omega_j\}_{j=1}^{\infty}$ and $\{\varpi_j\}_{j=1}^{\infty}$ be a smooth basis of $\textbf{W}^{1,2}(\Omega)$ orthonormal in the space $\textbf{L}^2(\Omega)$
such that $(\omega_i,\omega_j)=\delta_{i,j}$ and $(\varpi_i,\varpi_j)=\delta_{i,j}$. For $\epsilon>0$ and fixed $k\in\textbf{N}$ we seek approximation solutions $(\rho^{k,\epsilon},\textbf{u}^{k,\epsilon},\theta^{k,\epsilon},\textbf{H}^{k,\epsilon})$ where $(\textbf{u}^{k,\epsilon},\theta^{k,\epsilon},\textbf{H}^{k,\epsilon})$ has the form
\begin{eqnarray*}
\textbf{u}^{k,\epsilon}=\sum_{j=1}^ka_j^{k,\epsilon}\psi_j,~~\theta^{k,\epsilon}=\sum_{j=1}^kb_j^{k,\epsilon}\omega_j,~~\textbf{H}^{k,\epsilon}=\sum_{j=1}^kc_j^{k,\epsilon}\varpi_j,
\end{eqnarray*}
and solve the following approximation system (for all $1\leq j\leq k$ and a.e. $t\in(0,T)$)
\begin{eqnarray}\label{E6-1}
\rho^{k,\epsilon}_t+\textbf{div}(\rho^{k,\epsilon}\textbf{u}^{k,\epsilon})-\epsilon\triangle\rho^{k,\epsilon}=0,~~on~~[0,T]\times\Omega,
\end{eqnarray}
\begin{eqnarray}\label{E6-2}
(\rho^{k,\epsilon}\frac{d}{dt}\textbf{u}^{k,\epsilon},\psi_j)+(\rho^{k,\epsilon}[\textbf{u}^{k,\epsilon}]\textbf{u}^{k,\epsilon},\psi_j)&+&(\textbf{S}_{k,\epsilon},\textbf{D}(\psi_j))-\epsilon(\nabla\rho^{k,\epsilon},[\nabla\textbf{u}^{k,\epsilon}]\psi_j)\nonumber\\
&=&((\nabla\times\textbf{H}^{k,\epsilon})\times\textbf{H}^{k,\epsilon},\psi_j),
\end{eqnarray}
\begin{eqnarray}\label{E6-3}
(\rho^{k,\epsilon}\frac{d}{dt}Q(\theta^{k,\epsilon}),\omega_j)&+&(\rho^{k,\epsilon}Q(\theta^{k,\epsilon})\textbf{u}^{k,\epsilon},\omega_j)+(\textbf{q}(\rho_{k,\epsilon},\theta^{k,\epsilon},\nabla\theta^{k,\epsilon}),\nabla \omega_j))\nonumber\\
&=&(|\nabla\times\textbf{H}^{k,\epsilon}|^2,\omega_j)+(\textbf{S}_{k,\epsilon},\textbf{D}(\textbf{u}^{k,\epsilon})\omega_j),
\end{eqnarray}
\begin{eqnarray}\label{E6-4}
(\frac{d}{dt}\textbf{H}^{k,\epsilon},\varpi_j)+\nu(\textbf{curl}\textbf{H}^{k,\epsilon},\textbf{curl}\varpi_j)
-((\textbf{u}^{k,\epsilon}\times\textbf{H}^{k,\epsilon},\textbf{curl}\varpi_j)=0,
\end{eqnarray}
with
\begin{eqnarray*}
&&\rho^{k,\epsilon}\cdot\textbf{n}=0~~on~~[0,T],~~\rho_*\leq\rho^{k,\epsilon}\leq\rho^*,~~\rho^{k,\epsilon}(0,\cdot)=\rho_0~~in~~(0,T)\times\Omega,\\
&&\textbf{u}^{k,\epsilon}(0,\cdot)=\textbf{u}^{k,\epsilon}_0=\Gamma^k\textbf{u}_0,~~\theta^{k,\epsilon}(0,\cdot)=\theta^{k,\epsilon}_0=\Gamma^k\theta_0,
~~\textbf{H}^{k,\epsilon}(0,\cdot)=\textbf{H}^{k,\epsilon}_0=\Gamma^k\textbf{H}_0,
\end{eqnarray*}
where
\begin{eqnarray*}
\textbf{S}_{k,\epsilon}:=\textbf{S}(\rho^{k,\epsilon},\theta_{max}^{k,\epsilon},\textbf{D}^{k,\epsilon})~~with~~\theta_{max}^{k,\epsilon}:=\max\{\theta^{k,\epsilon},0\}.
\end{eqnarray*}
The existence of solutions for the approximation system (\ref{E6-1})-(\ref{E6-4}) can be proved by modified Faedo-Galerkin method. We first give the solvable of induction equation (\ref{E6-4}).
\begin{lemma}
Assume that the initial data $\textbf{H}^{k,\epsilon}(0)\in\textbf{Y}^k$ and
given a velocity field
$\textbf{u}^{k,\epsilon}\in\textbf{C}([0,T],\textbf{X}^k)$. The system
(\ref{E6-4}) has a solution
$\textbf{H}^{k,\epsilon}(x,t)\in\textbf{C}^1([0,T];\textbf{W}^{1,2})$. Moreover, the
operator $\textbf{u}^{k,\epsilon}\rightarrow\textbf{H}^{k,\epsilon}(\textbf{u}^{k,\epsilon})$ maps
bounded sets in $\textbf{C}([0,T],\textbf{X}^k)$ into bounded
subsets of $\textbf{Y}^k$, and the solution operator is continuous
operator.
\end{lemma}
\begin{proof}
Define
\begin{eqnarray*}
\textbf{X}^k:=\{\textbf{u}^{k,\epsilon}|\textbf{u}^{k,\epsilon}(x,t)=\sum_{j=1}^ka_j^{k,\epsilon}(t)\psi_j(x)\},~~
\textbf{Y}^k:=\{\textbf{H}^{k,\epsilon}|\textbf{H}^{k,\epsilon}(x,t)=\sum_{j=1}^kc_j^{k,\epsilon}(t)\varpi_j(x)\}.
\end{eqnarray*}
Note that $\textbf{H}^{k,\epsilon}(t,x)\in\textbf{Y}^k$. We can write
\begin{eqnarray*}
\textbf{H}^{k,\epsilon}=\sum_{j=1}^kc_j^{k,\epsilon}(t)\varpi_j(x),
\end{eqnarray*}
where the coefficients $c_j^{k,\epsilon}(t)$ are required to solve the
system of ordinary differential equations
\begin{eqnarray}\label{ER3-6}
\frac{dc_j^{k,\epsilon}}{dt}+\sum_{|j|\leq
k}A_{i,j}(t)c_j^{k,\epsilon}=0,~~|j|\leq k,
\end{eqnarray}
where
\begin{eqnarray*}\label{ER3-7}
A_{j,k}(t)&=&\nu(\nabla\varpi_i,\nabla\varpi_j)-((\textbf{u}^{k,\epsilon}\cdot\nabla)\varpi_i,\varpi_j)
-((\varpi_i\cdot\nabla)\textbf{u}^{k,\epsilon},\varpi_j)\nonumber\\
&&+((\textbf{div}\textbf{u}^{k,\epsilon})\varpi_i,\varpi_j).
\end{eqnarray*}
For given initial data
$\textbf{H}^{k,\epsilon}(0)\in\textbf{Y}^k$, the system (\ref{ER3-6}) has a
unique solution $c_j^{k,\epsilon}\in\textbf{C}^1((0,T);\textbf{Y}^k)$
for some $T'\leq T$. Multiplying both sides (\ref{ER3-6}) by
$c_j^{k,\epsilon}$, summing over $j$, integrating by parts, we have
\begin{eqnarray}\label{ER3-8}
\frac{d}{dt}\|\textbf{H}^{k,\epsilon}\|^2_{\textbf{L}^2(\Omega)}+2\nu\|\nabla\textbf{H}^{k,\epsilon}\|^2_{\textbf{L}^2(\Omega)}
&=&-2(\textbf{H}^{k,\epsilon}\cdot\nabla\textbf{u}^{k,\epsilon},\textbf{H}^{k,\epsilon}).\nonumber
\end{eqnarray}
Note that by Young inequality,
\begin{eqnarray}\label{ER3-9}
-(\textbf{H}_n\cdot\nabla\textbf{u}^{k,\epsilon},\textbf{H}^{k,\epsilon})\leq\frac{C\nu}{2}\|\textbf{H}^{k,\epsilon}\|^2_{\textbf{L}^2(\Omega)}+\frac{2}{C\nu}\|\nabla\textbf{u}^{k,\epsilon}\|^2_{\textbf{L}^2(\Omega)},
\end{eqnarray}
Using the Poincar\'{e} inequality,
\begin{eqnarray}\label{ER3-11}
\|\nabla\textbf{H}^{k,\epsilon}\|_{\textbf{L}^2(\Omega)}\geq
C\|\textbf{H}^{k,\epsilon}\|_{\textbf{L}^2(\Omega)}.
\end{eqnarray}
By (\ref{ER3-8})-(\ref{ER3-11}), we derive
\begin{eqnarray*}
\frac{d}{dt}\|\textbf{H}^{k,\epsilon}\|^2_{\textbf{L}^2(\Omega)}+C\nu\|\textbf{H}^{k,\epsilon}\|^2_{\textbf{L}^2(\Omega)}
\leq\frac{2}{C\nu}\|\nabla\textbf{u}^{k,\epsilon}\|_{\textbf{L}^2(\Omega)}.
\end{eqnarray*}
Thus, for $t\in[0,T]$, we get
\begin{eqnarray}\label{ER3-12}
\|\textbf{H}^{k,\epsilon}\|^2_{\textbf{L}^2(\Omega)}&\leq&\|\textbf{H}^{k,\epsilon}(0)\|^2_{\textbf{L}^2(\Omega)}e^{-C\nu t}+\frac{2}{C\nu}\int_0^te^{-C\nu(t-s)}\|\nabla\textbf{u}^{k,\epsilon}\|_{\textbf{L}^2(\Omega)}ds.~~~~~~~
\end{eqnarray}
This implies that for $t\in[0,T]$,
\begin{eqnarray*}
\|\textbf{H}^{k,\epsilon}(t)\|^2_{\textbf{L}^2(\Omega)}
\leq\|\textbf{H}^{k,\epsilon}(0)\|^2_{\textbf{L}^2(\Omega)}e^{-C\nu t}+\frac{2}{C\nu}\int_0^{\omega}e^{-C\nu(t-s)}\|\nabla\textbf{u}^{k,\epsilon}\|_{\textbf{L}^2(\Omega)}ds.
\end{eqnarray*}
Due to $\{\varpi_j(x)\}_{j=1}^{\infty}$ be an orthonormal basis of
$\textbf{W}^{1,2}(\Omega)$, so we have
$|c^{k,\epsilon}_j(t)|=\|\textbf{H}^{k,\epsilon}(t)\|^2_{\textbf{L}^2(\Omega)}$,
from which we conclude that $T'=T$.

Define the ball $B_{R}$ of radius $R$ and the map
$\Pi:B_{R}\longrightarrow B_{R}$ such that
$\Pi(\textbf{H}^{k,\epsilon}(0))=\textbf{H}^{k,\epsilon}(T)$, where the radius $R$
such that
\begin{eqnarray*}
R\geq\left(\frac{\frac{2}{C\nu}\int_0^{T}e^{-C\nu(t-s)}\|\nabla\textbf{u}^{k,\epsilon}\|_{\textbf{L}^2(T)}ds}{1-e^{-C\nu T}}\right)^{\frac{1}{2}}.
\end{eqnarray*}
Follows \cite{Lions1}, we can prove the map
$\Pi$ is continuous. Hence, it has a fixed point. Moreover, from
(\ref{ER3-12}), we know that the solution operator
$\textbf{u}^{k,\epsilon}\longrightarrow\textbf{H}^{k,\epsilon}(\textbf{u}^{k,\epsilon})$ maps bounded
sets in $\textbf{C}([0,T],\textbf{X}^k)$ into bounded subsets of the
set $\textbf{Y}^k$. Then, as done in \cite{Hu2}, the solution
operator $\textbf{u}^{k,\epsilon}\longrightarrow\textbf{H}^{k,\epsilon}(\textbf{u}^{k,\epsilon})$ is
a continuity operator. This completes the proof.
\end{proof}
The rest process of proof is similar to Proposition 7.2. in \cite{Fei1} or Lemma 3.2. in \cite{Hu2}. Thus, combining with Lemma 4, we obtain the following result:
\begin{lemma}
Under the assumption in Theorem 1, for fixed $k\in\textbf{N}$ and $\epsilon>0$, the approximation problem (\ref{E6-1})-(\ref{E6-4}) has a solution $(\rho^{k,\epsilon},\textbf{u}^{k,\epsilon},\theta^{k,\epsilon},\textbf{H}^{k,\epsilon})$ on $(0,T)\times\Omega$ for any fixed $T>0$.
\end{lemma}

\subsection{Limit $\epsilon\longrightarrow0$}

First, we summarize the estimates available for (\ref{E6-1})-(\ref{E6-4}) for $\epsilon>0$ and $k\in\textbf{N}$ fixed. Then the behavior of relevant solutions will be studied as $\epsilon\longrightarrow0$.

Multiplying (\ref{E6-1}) by $\rho^{k,\epsilon}$ leads to
\begin{eqnarray}\label{E6-5R}
\sup_{t\in[0,T]}\|\rho^{k,\epsilon}\|^2_{\textbf{L}^2(\Omega)}+2\epsilon\int_0^T\|\nabla\rho^{k,\epsilon}\|^2_{\textbf{L}^2(\Omega)}dt\leq\|\rho_0\|^2_{\textbf{L}^2(\Omega)}.
\end{eqnarray}
It follows from (\ref{E1-13}) and a weak maximum (minimum) principle that
\begin{eqnarray}\label{E6-5R1}
\rho_*\leq\rho^{k,\epsilon}\leq\rho^*.
\end{eqnarray}
Taking the $\textbf{L}^2$ scalar product of (\ref{E6-1}) with a smooth $z$ leads to the equation
\begin{eqnarray}\label{E6-5}
\langle\rho^{k,\epsilon}_t,z\rangle-(\rho^{k,\epsilon}\textbf{u}^{k,\epsilon},\nabla z)+\epsilon(\nabla\rho^{k,\epsilon},\nabla z)=0.
\end{eqnarray}
Multiplying the $j$th equation in (\ref{E6-2}) by $a_j^{k,\epsilon}$, then taking the sum over $j=1,\ldots,n$, using (\ref{E6-5}) with $z=\frac{|\textbf{u}^{k,\epsilon}|^2}{2}$ and integrating the equality over $(0,t)$, we have
\begin{eqnarray}\label{E6-7}
\|\textbf{u}^{k,\epsilon}\|_{\textbf{L}^2(\Omega)}^2&+&\|\sqrt{\rho^{k,\epsilon}}\textbf{u}^{k,\epsilon}\|_{\textbf{L}^2(\Omega)}^2
+4\int_0^t(\textbf{S}^{k,\epsilon},\textbf{D}(\textbf{u}^{k,\epsilon}))ds\nonumber\\
&&+\int_0^t\int_{\Omega}(\textbf{H}^{k,\epsilon})^T\nabla\textbf{u}^{k,\epsilon}\textbf{H}^{k,\epsilon}
+\frac{1}{2}\nabla(|\textbf{H}^{k,\epsilon}|^2)\textbf{u}^{k,\epsilon}dxds\nonumber\\
&&\leq2\|\sqrt{\rho_0}\Gamma^n\textbf{u}_0\|_{\textbf{L}^2(\Omega)}^2.
\end{eqnarray}
Multiplying the $j$th equation in (\ref{E6-4}) by $c_j^{k,\epsilon}$, then taking the sum over $j=1,\ldots,n$ and integrating the equality over $(0,t)$, we have
\begin{eqnarray}\label{E6-6}
2\|\textbf{H}^{k,\epsilon}\|^2_{\textbf{L}^2(\Omega)}&+&4\int_0^t\nu\|\nabla\times\textbf{H}^{k,\epsilon}\|_{\textbf{L}^2(\Omega)}^2ds\nonumber\\
&&-4\int_0^t\int_{\Omega}((\textbf{H}^{k,\epsilon})^T\nabla\textbf{u}^{k,\epsilon}\textbf{H}^{k,\epsilon}+\frac{1}{2}\nabla(|\textbf{H}^{k,\epsilon}|^2)\cdot\textbf{u}^{k,\epsilon})dx)ds\nonumber\\
&=&2\|\Gamma^n\textbf{H}_0\|_{\textbf{L}^2(\Omega)}^2.
\end{eqnarray}
Summing up (\ref{E6-7})-(\ref{E6-6}), we obtain
\begin{eqnarray}\label{E6-8}
\|\textbf{u}^{k,\epsilon}\|_{\textbf{L}^2(\Omega)}^2&+&\|\sqrt{\rho^{k,\epsilon}}\textbf{u}^{k,\epsilon}\|_{\textbf{L}^2(\Omega)}^2+2\|\textbf{H}^{k,\epsilon}\|^2_{\textbf{L}^2(\Omega)}\nonumber\\
&+&4\int_0^t\left((\textbf{S}^{k,\epsilon},\textbf{D}(\textbf{u}^{k,\epsilon})+\nu\|\nabla\times\textbf{H}^{k,\epsilon}\|_{\textbf{L}^2(\Omega)}^2\right)ds\nonumber\\
&\leq&2\|\sqrt{\rho_0}\Gamma^n\textbf{u}_0\|_{\textbf{L}^2(\Omega)}^2+2\|\textbf{H}_0\|_{\textbf{L}^2(\Omega)}^2.
\end{eqnarray}
Using Korn's inequality to the fourth term we get
\begin{eqnarray}\label{E6-9}
\sup_{t\in[0,T]}(\|\textbf{u}^{k,\epsilon}\|_{\textbf{L}^2(\Omega)}^2+\|\textbf{H}^{k,\epsilon}\|^2_{\textbf{L}^2(\Omega)})+\int_0^T(\|\textbf{u}^{k,\epsilon}\|_{\textbf{W}^{1,r}(\Omega)}^r&+&\|\nabla\times\textbf{H}^{k,\epsilon}\|^2_{\textbf{L}^2(\Omega)})dt\nonumber\\
&&\leq C.
\end{eqnarray}
Multiplying the $j$th equation in (\ref{E6-3}) by $b_j^{k,\epsilon}$, then taking the sum over $j=1,\ldots,n$, using (\ref{E6-5}) with $z=\frac{|\theta^{k,\epsilon}|^2}{2}$ and integrating the equality over $(0,t)$, we have
\begin{eqnarray}\label{E6-10}
\sup_{t\in[0,T]}\{\|\theta^{k,\epsilon}\|_{\textbf{L}^2(\Omega)}^2&+&\|\sqrt{\rho^{k,\epsilon}}\theta^{k,\epsilon}\|_{\textbf{L}^2(\Omega)}^2\}
+\int_0^T\|\sqrt{\kappa^{k,\epsilon}}\nabla\theta^{k,\epsilon}\|^2_{\textbf{L}^2(\Omega)}dt\nonumber\\
&\leq&C\|\sqrt{\rho_0}\theta_0^n\|_{\textbf{L}^2(\Omega)}^2+C\int_0^T\|\textbf{S}^{k,\epsilon}\cdot\textbf{D}(\textbf{u}^{k,\epsilon})\|^2_{\textbf{L}^2(\Omega)}dt\nonumber\\
&&+C\int_0^T\|\nabla\times\textbf{H}^{k,\epsilon}\|^2_{\textbf{L}^2(\Omega)}dt\leq C(n).
\end{eqnarray}
For fixed $k\in\textbf{N}$ we can multiply the $j$th equation in (\ref{E6-2}) by $\frac{da_j^{k,\epsilon}}{dt}$, the $j$th equation in (\ref{E6-3}) by $\frac{db_j^{k,\epsilon}}{dt}$ and the $j$th equation in (\ref{E6-4}) by $\frac{dc_j^{k,\epsilon}}{dt}$. Then we obtain
\begin{eqnarray}\label{E6-11}
\int_0^T|\frac{da_j^{k,\epsilon}}{dt}|^2dt,~~\int_0^T|\frac{db_j^{k,\epsilon}}{dt}|^2dt,~~\int_0^T|\frac{dc_j^{k,\epsilon}}{dt}|^2dt\leq C(n).
\end{eqnarray}
By (\ref{E6-5R})-(\ref{E6-10}), we can have the following convergence result as $\epsilon\longrightarrow0$
\begin{eqnarray}\label{E6-12}
\rho^{k,\epsilon}\longrightarrow\rho^k~~*-weakly~in~\textbf{L}^{\infty}([0,T]\times\Omega),
\end{eqnarray}
\begin{eqnarray}\label{E6-13}
a^{k,\epsilon}\rightharpoonup a^k~~weakly~in~\textbf{W}^{1,2}(0,T)~and~strongly~in~\mathcal{C}([0,T]),
\end{eqnarray}
\begin{eqnarray}\label{E6-14}
b^{k,\epsilon}\rightharpoonup b^k~~weakly~in~\textbf{W}^{1,2}(0,T)~and~strongly~in~\mathcal{C}([0,T]),
\end{eqnarray}
\begin{eqnarray}\label{E6-15}
c^{k,\epsilon}\rightharpoonup c^k~~weakly~in~\textbf{W}^{1,2}(0,T)~and~strongly~in~\mathcal{C}([0,T]).
\end{eqnarray}
It follows from (\ref{E6-12}) and (\ref{E6-15}) that
\begin{eqnarray}\label{E6-16}
\textbf{u}^{k,\epsilon}\longrightarrow\textbf{u}^k~~strongly~in~\textbf{L}^{2r}(0,T;\textbf{W}^{1,2r}_n(\Omega)),
\end{eqnarray}
\begin{eqnarray}\label{E6-17}
\textbf{H}^{k,\epsilon}\longrightarrow\textbf{H}^k~~strongly~in~\textbf{L}^{2r}(0,T;\textbf{W}^{1,2r}_n(\Omega)).
\end{eqnarray}
By (\ref{E6-12}), (\ref{E6-16})-(\ref{E6-17}), we can take the limit in the weak formulation of (\ref{E6-1}) and get the transport equation
\begin{eqnarray*}
\rho^{k}_t+\textbf{div}(\rho^{k}\textbf{u}^{k})=0,~~on~~[0,T]\times\Omega.
\end{eqnarray*}
Then applying Diperna-Lions theory (see \cite{Diperna,Lions0}) of the renormalized solutions to the transport equation, we conclude that
\begin{eqnarray*}
\rho^{k,\epsilon}\longrightarrow\rho^k~~strongly~in~\textbf{L}^{2}([0,T];\textbf{L}^2(\Omega))~~a.e.~in~(0,T)\times\Omega.
\end{eqnarray*}
Thus by the above convergence results, we can take the limit $\epsilon\longrightarrow0$ and obtain the solution $(\rho^k,\textbf{u}^k,\theta^k,\textbf{H}^k)$ solving the following system
\begin{eqnarray}\label{E6-1R1}
\rho^{k}_t+\textbf{div}(\rho^{k}\textbf{u}^{k})=0,~~on~~[0,T]\times\Omega,
\end{eqnarray}
\begin{eqnarray}\label{E6-2R1}
(\rho^{k}\frac{d}{dt}\textbf{u}^{k},\psi_j)+(\rho^{k}[\textbf{u}^{k}]\textbf{u}^{k},\psi_j)+(\textbf{S}_{k},\textbf{D}(\psi_j))
=((\nabla\times\textbf{H}^{k})\times\textbf{H}^{k},\psi_j),~~~
\end{eqnarray}
\begin{eqnarray}\label{E6-3R1}
(\rho^{k}\frac{d}{dt}Q(\theta^{k}),\omega_j)+(\rho^{k}Q(\theta^{k})\textbf{u}^{k},\omega_j)&+&(\textbf{q}(\rho_{k},\theta^{k},\nabla\theta^{k}),\nabla \omega_j))\nonumber\\
&=&(|\nabla\times\textbf{H}^{k}|^2,\omega_j)+(\textbf{S}_{k},\textbf{D}(\textbf{u}^{k})\omega_j),~~~~~
\end{eqnarray}
\begin{eqnarray}\label{E6-4R1}
(\frac{d}{dt}\textbf{H}^{k},\varpi_j)+\nu(\textbf{curl}\textbf{H}^{k},\textbf{curl}\varpi_j)
-(\textbf{u}^{k}\times\textbf{H}^{k},\textbf{curl}\varpi_j)=0,~~~
\end{eqnarray}
with
\begin{eqnarray*}
&&\rho^{k}\cdot\textbf{n}=0~~on~~[0,T],~~\rho_*\leq\rho^{k}\leq\rho^*,~~\rho^{k}(0,\cdot)=\rho_0~~in~~(0,T)\times\Omega,\\
&&\textbf{u}^{k}(0,\cdot)=\Gamma^n\textbf{u}_0,~~\theta^{k}(0,\cdot)=\Gamma^k\theta_0,
~~\textbf{H}^{k}(0,\cdot)=\Gamma^k\textbf{H}_0,
\end{eqnarray*}
where
\begin{eqnarray*}
\textbf{S}_{k}:=\textbf{S}(\rho^{k},\theta_{max}^{k},\textbf{D}^{k})~~with~~\theta_{max}^{k}:=\max\{\theta^{k},0\}.
\end{eqnarray*}

\subsection{Limit $k\longrightarrow\infty$}In this subsection, we establish some uniform estimates with respect to $k$. We follow the method in \cite{B3,Freh2}.
Using the similar procedures of (\ref{E6-9}), (\ref{E6-10}), (\ref{E6-11}) and (\ref{E6-1R1}), we have
\begin{eqnarray}\label{E6-18}
\rho_*\leq\rho^{k}\leq\rho^*,~~\int_0^T\|\rho_t^{k}\|^{\frac{q}{q-1}}_{(\textbf{W}^{1,q}(\Omega))^*}dt\leq C~~\forall q\in(1,\infty),
\end{eqnarray}
\begin{eqnarray}\label{E6-18R1}
\sup_{t\in[0,T]}\{\|\textbf{u}^{k}\|_{\textbf{L}^2(\Omega)}^2&+&\|\textbf{H}^{k}\|^2_{\textbf{L}^2(\Omega)}\}\nonumber\\
&&+\int_0^T(\|\textbf{u}^{k}\|_{\textbf{W}^{1,r}(\Omega)}^r+\|\nabla\times\textbf{H}^{k}\|^2_{\textbf{L}^2(\Omega)})dt\leq C,
\end{eqnarray}
\begin{eqnarray}\label{E6-19}
\sup_{t\in[0,T]}\{\|\theta^{k}\|_{\textbf{L}^2(\Omega)}^2+\|\sqrt{\rho^{k}}\theta^{k}\|_{\textbf{L}^2(\Omega)}^2\}
+\int_0^T\|\sqrt{\kappa^{k}}\nabla\theta^{k}\|^2_{\textbf{L}^2(\Omega)}dt\leq C(n),~~~~~~~
\end{eqnarray}
\begin{eqnarray}\label{E6-18R2}
\int_0^T|\frac{da_j^{k}}{dt}|^2dt,~~\int_0^T|\frac{db_j^{k}}{dt}|^2dt,~~\int_0^T|\frac{dc_j^{k}}{dt}|^2dt\leq C(n).
\end{eqnarray}
Using Lions-Diperna theory of renormalized solutions, by (\ref{E6-18}), we have
\begin{eqnarray}\label{E6-19R}
\rho^k\longrightarrow\rho~~strongly~~in~~\textbf{C}([0,T];\textbf{L}^p(\Omega))~~a.e.~in~(0,T)\times\Omega.
\end{eqnarray}
By (\ref{E6-18R1}) and (\ref{E6-18R2}), we get
\begin{eqnarray*}
\textbf{u}^k\longrightarrow\textbf{u}~~strongly~~in~~\textbf{L}^{2r}(0,T;\textbf{W}_n^{1,2r}(\Omega)),
\end{eqnarray*}
and
\begin{eqnarray*}
\textbf{H}^k\longrightarrow\textbf{H}~~strongly~~in~~\textbf{L}^{2}(0,T;\textbf{W}_n^{1,2}(\Omega)).
\end{eqnarray*}
It follows from (\ref{E6-18})-(\ref{E6-19}) that
\begin{eqnarray}\label{E6-20}
\sup_{t\in[0,T]}\{\|(\rho^k\theta^{k})(t)\|_{\textbf{L}^2(\Omega)}^2\}&+&\int_{Q}|\nabla\theta^k|^2dxds\nonumber\\
&&+\int_{((0,T)\times\Omega)/Q}(\theta^k)^{\alpha}|\nabla\theta^k|^2dxds\leq C(n),~~
\end{eqnarray}
where $Q=\{(t,x)\in(0,T)\times\Omega;\theta^k(t,x)\leq\theta_*\}$.

Define
\begin{eqnarray*}
\bar{\kappa}(\theta):=\left\{
\begin{array}{lll}
&&\theta^{\alpha}~~for~\theta\geq\theta_*,\\
&&\theta_*^{\alpha}~~for~\theta<\theta_*,
\end{array}
\right.
\end{eqnarray*}
and
\begin{eqnarray*}
\overline{K}(\theta):=\left\{
\begin{array}{lll}
&&\frac{2}{\alpha+2}\theta^{\frac{\alpha+2}{2}}+\frac{\alpha}{\alpha+2}\theta_*^{\frac{\alpha+2}{2}}~~for~\theta\geq\theta_*,\\
&&\theta_*^{\frac{\alpha}{2}}\theta~~for~\theta\leq\theta_*.
\end{array}
\right.
\end{eqnarray*}
Using (\ref{E6-20}), we have
\begin{eqnarray}\label{E6-21}
\sup_{t\in(0,T)}\|\theta^k\|^2_{\textbf{L}^2(\Omega)}+\int_0^T|\nabla\overline{K}(\theta^k)|^2dt\leq C(n).
\end{eqnarray}
Using the similar procedure in \cite{B3} (also see \cite{Freh2}), by (\ref{E6-21}), we get
\begin{eqnarray}\label{E6-21R1}
\sup_{t\in(0,T)}\|\overline{K}(\theta^k)\|^2_{\textbf{L}^2(\Omega)}+\int_0^T\|\overline{K}(\theta^k)\|^2_{\textbf{W}^{1,2}(\Omega)}dt\leq C(n),
\end{eqnarray}
\begin{eqnarray*}
\int_0^T\|\textbf{q}^k\|^m_{\textbf{L}^m(\Omega)}dt\leq C(n)~~with~m=2~for~\alpha\leq0~~and~~m=\frac{3\alpha+10}{3\alpha+5}~for~\alpha>0,
\end{eqnarray*}
\begin{eqnarray}\label{E6-21R2}
\|(\bar{\kappa}(\theta^k))^{-\frac{1}{2}}\kappa_k\|_{\textbf{L}^q((0,T)\times\Omega)}\leq C(n)&&with~q=\infty~for~\alpha\leq0\nonumber\\
&&and~~q=\frac{2(3\alpha+10)}{3\alpha}~for~\alpha>0,~~~~
\end{eqnarray}
\begin{eqnarray*}
\int_0^T|\nabla\theta^k|^qdxdt\leq C(n)~~with~q=\frac{5(\alpha+2)}{\alpha+5}~for~\alpha\leq0~~and~~q=2~for~\alpha>0.
\end{eqnarray*}
Then it follows from the above estimates, (\ref{E6-3R1}) and the continuity of the projection $\Gamma^k$ that
\begin{eqnarray}\label{E6-22}
\|\partial_t(\rho^k\theta^k)\|_{\textbf{L}^{q'}(0,T;\textbf{W}^{1,\delta'}(\Omega))}\leq C(n)~~for~~\delta=\{2,\frac{3\alpha+10}{3\alpha+5}\}.
\end{eqnarray}
By (\ref{E6-19})-(\ref{E6-22}), we can establish the following convergence results as $k\longrightarrow\infty$
\begin{eqnarray}\label{E6-22R}
\theta^k\rightharpoonup\theta~~weakly~in~\textbf{L}^q(0,T;\textbf{W}^{1,q}(\Omega)),
\end{eqnarray}
\begin{eqnarray*}
\rho^k\theta^k\rightharpoonup\rho\theta~~*-weakly~in~\{z\in\textbf{L}^{\infty}(0,T;\textbf{L}^{2}(\Omega)),z_t\in\textbf{L}^{\delta'}(0,T;\textbf{W}^{-1,\delta'}(\Omega))\}.
\end{eqnarray*}
Consequently, using Aubin-Lions Lemma 1, we have
\begin{eqnarray*}
\rho^k\theta^k\longrightarrow\rho\theta~~strongly~in~\textbf{C}(0,T;(\textbf{W}^{1,q}(\Omega))^*).
\end{eqnarray*}
Then
\begin{eqnarray*}
\lim_{k\longrightarrow\infty}\int_0^T(\rho^k\theta^k,\theta^k)dt&=&\lim_{k\longrightarrow\infty}\int_0^T\langle\rho^k\theta^k,\theta^k\rangle_{(\textbf{W}^{1,q}(\Omega))^*} dt\nonumber\\
&=&\int_0^T\langle\rho\theta,\theta\rangle_{(\textbf{W}^{1,q}(\Omega))^*}=\int_0^T(\rho\theta,\theta)dt,
\end{eqnarray*}
which together with (\ref{E6-22R}) shows that
\begin{eqnarray}\label{E6-23}
\theta^k\longrightarrow\theta~~strongly~in~\textbf{L}^2((0,T)\times\Omega)~~a.e.~in~(0,T)\times\Omega,
\end{eqnarray}
This combining with (\ref{E6-19R}), (\ref{E6-21R1})-(\ref{E6-21R2}) implies that
\begin{eqnarray*}
\overline{K}(\theta^k)\rightharpoonup\overline{K}(\theta)~~weakly~~in~~\textbf{L}^2(0,T;\textbf{W}^{1,2}(\Omega)),
\end{eqnarray*}
\begin{eqnarray*}
(\bar{\kappa}(\theta^k))^{-\frac{1}{2}}\kappa_k\longrightarrow(\bar{\kappa}(\theta))^{-\frac{1}{2}}\kappa~~strongly~~in~~\textbf{L}^{q^*}(0,T;\textbf{L}^{q^*}(\Omega))~~\forall~q^*<q,
\end{eqnarray*}
and hence
\begin{eqnarray*}
\textbf{q}_k\rightharpoonup\textbf{q}:=\bar{\kappa}(\rho,\theta)\nabla\theta~~weakly~~in~~\textbf{L}^{\delta}(0,T;\textbf{W}^{1,\delta}(\Omega)),
\end{eqnarray*}
where $\bar{\kappa}:=\kappa(\rho,\theta_{max})$ with $\theta_{max}:=\{\theta,\theta_*\}$.

Therefore, the above convergence results allow us to take the limit in (\ref{E6-1R1})-(\ref{E6-4R1}) and to obtain
\begin{eqnarray*}
\int_{0}^T\langle\rho_t,z\rangle-(\rho\textbf{u},\nabla z)dt=0,~~\rho_*\leq\rho\leq\rho^*,
\end{eqnarray*}
for any $z\in\textbf{L}^{q}(0,T;\textbf{W}^{1,s}(\Omega))$ for any $q\in[1,\infty)$.
\begin{eqnarray*}
\langle(\rho\textbf{u})_t,\varphi_j\rangle-(\rho\textbf{u}\otimes\textbf{u},\nabla\varphi_j)&+&(\textbf{S}(\rho,\theta,\textbf{D}(\textbf{u})),\textbf{D}(\varphi_j))\nonumber\\
&&+\int_{\Omega}\textbf{H}^T\nabla\varphi_j\textbf{H}+\frac{1}{2}\nabla(|\textbf{H}|^2)\cdot\varphi_jdx=0,
\end{eqnarray*}
\begin{eqnarray*}
\int_0^T(\langle(\rho Q(\theta))_t,h\rangle&-&(\rho Q(\theta)\textbf{u},\nabla h)-(\textbf{q}(\rho,\theta,\nabla\theta),\nabla h))dt\nonumber\\
&=&\int_0^T\left(\nu(|\nabla\times\textbf{H}|^2,h)+(\textbf{S}(\rho,\theta,\textbf{D}(\textbf{u})),\textbf{D}(\textbf{u})h)\right)dt,
\end{eqnarray*}
for all $h\in\textbf{L}^{\infty}(0,T;\textbf{W}^{1,q})$ with $q$ sufficiently large,
\begin{eqnarray*}
\int_0^T\left(\langle\textbf{H}_t,b\rangle+\nu(\textbf{curl}\textbf{H},\textbf{curl}b)
+(\textbf{u}\times\textbf{H},\textbf{curl}b)\right)dt=0,
\end{eqnarray*}
for all $b\in\textbf{L}^{2}(0,T;\textbf{W}^{1,2}(\Omega))$.

\begin{acknowledgements}
This work was done when the author visited University Kansas in 2011.
The author is grateful for the hospitality of the department of mathematics, University of Kansas,
and express his sincerely thanks to prof Weishi. Liu for his help!
This work is supported by NSFC No 11201172, SRFDP Grant No 20120061120002 and the 985 Project of Jilin University.

\end{acknowledgements}


\begin{thebibliography}{}
%
%
\bibitem{Andra}
Andra, W., Nowak, H.: Magnetism in Medicine, Wiley VCH, Berlin, 1998.

\bibitem{Bar}
Barnothy (Ed.), M.F.: Biological Effects of Magnetic Fields, Plenum Press, New York, 1964.

\bibitem{Bh}
Bhargava, R., Sugandha, Takhar, H.S., Beg, O.A.: Computational simulation of biomagnetic micropolar blood flow in porous media, J. Biomech. \textbf{39}
S648-S649 (2006)

\bibitem{Br}
Brown, R.M., Shen, Z.: Estimates for the Stokes problem operator in Lipschitz domains. Indiana Univ. Math. J. \textbf{44}, 1183-1206 (1995)

\bibitem{B1}
Bul\'{\i}\v{c}ek, M., Feireisl, E., M\'{a}lek, J.: Navier-Stokes-Fourier system for incompressible fluids with temperature dependent material coefficients. Nonlinear Anal. Real World Appl. \textbf{10}, 992-1015 (2009)

\bibitem{B2}
Bul\'{\i}\v{c}ek, M., M\'{a}lek, J., Rajagopal, K.R.: Navier's slip and evolutionary Navier-Stokes like systems with pressure and shear-rate dependent viscosity. Indiana Univ. Math. J. \textbf{56}, 51-85 (2007)

\bibitem{B3}
Bul\'{\i}\v{c}ek, M., M\'{a}lek, J., Rajagopal, K.R.: Mathematical analysis of unsteady flows of fluids with pressure, shear-rate and temperature dependent material moduli, that slip at solid boundaries. SIAM J. Math. Anal. \textbf{41}, 665-707 (2009)

\bibitem{Caban}
Cabannes, H.: Theoretical Magnetofluiddynamics. New York: Academic Press, 1970

\bibitem{Diening}
Diening, L., Ru\v{z}i\v{c}ka, M., Wolf, J.: Existence of weak solutions for unsteady motions of generalized Newtonian fluids. Annali della Scuola Normale Superiore di Pisa IX:1-46. (2010)

\bibitem{Diperna}
DiPerna, R.J., Lions, P.L.: Ordinary differential equations, transport theory and Sobolev spaces. Invent.
Math. \textbf{98}, 511-547 (1989)

\bibitem{Fei4}
Ducomet, B., Feireisl, E.: The equations of Magnetohydrodynamics: on the interaction between matter and radiation in the evolution of gaseous stars.
Commun. Math. Phys. \textbf{226}, 595-629 (2006)

\bibitem{Duv}
Duvaut, G., Lions, J.L.: In\'{e}quation en thermo\'{e}lasticit\'{e}
et magn\'{e}to-hydrodynamique. Arch. Rational Mech. Anal.
\textbf{46}, 241-279 (1972)

\bibitem{El}
El-Shehawey, E.F.,  Elbarbary, E.M.E., Afifi, N.A.S., Elshahed, M.: MHD flow of an elastico-viscous fluid under periodic body acceleration. Int. J. Math. Math.
Sci. \textbf{23}  795-799 (2000)

\bibitem{Fei1}
Feireisl, E.: Dynamics of viscous compressible fluids. Oxford Lecture Series in Mathematics and its
Applications, 26. Oxford: Oxford University Press, 2004

\bibitem{Fern}
Fern\'{a}ndez-Cara, E., Guill\'{e}n, F., Ortega, R.R.: Some theoretical results for viscoplastic and dilatant fluids with variable desity. Nonlinear Anal. \textbf{28}, 1079-1100 (1997)

\bibitem{Freh1}
Frehse, J., Ru\v{z}i\v{c}ka, M.: Non-homogeneous generalized newtonian fluids. Math. Z. \textbf{260}, 355-375 (2008)

\bibitem{Freh2}
Frehse, J.,  M\'{a}lek, J., Ru\v{z}i\v{c}ka, M.: Large data existence result for unsteady flows of inhomogeneous shear thickening heat conducting incompressible fluids. Comm. PDE. \textbf{35}, 1891-1919 (2010)

\bibitem{Gal}
Galdi, G.P., Simader, C.G., Sohr, H.: On the Stokes problem in Lipschitz domains. Ann. Mat. Pura Appl. \textbf{167}, 147-163 (1994)

\bibitem{G}
Guill\'{e}n-Gonz\'{a}lez, F.: Density dependent incompressible fluids with non-Newtonian viscosity. Czechoslovak Math. J. \textbf{54}, 637-656 (2004)

\bibitem{Hu1}
Hu, X., Wang, D.: Global solutions to the three dimensional full
compressible magnetohydrodynamic flows. Comm. Math. Phys
\textbf{283}, 255-284 (2008)

\bibitem{Hu0}
Hu, X., Wang, D.: Compactness of weak solutions to the three-dimensional compressible magnetohydrodynamic equations,
J. Differential Equations. \textbf{245} 2176-2198 (2008)

\bibitem{Hu2}
Hu, X., Wang, D.: Global existence and large time behavior of
solutions to the three dimensional equations of compressible
magnetohydrodynamic flows. Arch. Rational Mech. Anal.
\textbf{197}, 203-238 (2010)

\bibitem{Hu3}
Hu, X., Wang, D.: Low mach number limit of viscous compressible
magnetohydrodynamic flows. SIAM J. Math. Anal. \textbf{41},
1272-1294 (2009)

\bibitem{Jiang}
Jiang, S., Ju, Q.C., Li, F.C.: Incompressible limit of the
compressible Magnetohydrodynamic equations with periodic boundary
conditions. Comm. Math. Phys. \textbf{297}, 371-400 (2010)

\bibitem{Jiang1}
Jiang, S., Ju, Q.C., Li, F.C.:  Incompressible limit of the
compressible Magnetohydrodynamic equations with vanishing viscosity coefficients.
SIAM J. Math. Anal. \textbf{42}, 2539-2553 (2010)

\bibitem{Kan}
Kantrovits, A.R., Petchek, G.Y.: Magnitnaya gidrodinamika (Magnetohydrodynamics). Atomizdat, Moscow, 1958

\bibitem{Ka}
Kawashima, S., Shelukhin, V.V.: Unique global solution with respect to time of initial boundary value problems for one dimensional equations of a viscous gas. J. Appl. Math. Mech. \textbf{41}, 273-282 (1977)

\bibitem{Kuli}
Kulikovskiy, A.G., Lyubimov, G.A.: Magnetohydrodynamics. Reading, MA: Addison-Wesley, 1965

\bibitem{Laudau}
Laudau, L.D., Lifshitz, E.M.: Electrodynamics of Continuous Media, 2nd ed., New York: Pergamon,
1984

\bibitem{Lions0}
Lions, P.L.: Mathematical topics in fluid mechanics. Vol. 1. Incompressible models. Oxford Lecture Series
in Mathematics and its Applications, Vol.3. New York: Oxford University Press, 1998

\bibitem{Lions1}
Lions, P.L.: Mathematical topics in fluid mechanics. Vol. 2. Compressible models. Oxford Lecture Series
in Mathematics and its Applications, 10. Oxford Science Publications. New York: The Clarendon Press,
Oxford University Press, 1998

\bibitem{Novo}
Novotn\'{y} A., Stra\v{s}kraba, I.: Introduction to the theory of compressible flow. Oxford: Oxford University
Press, 2004

\bibitem{Constan}
Ohkitani, K., Constantin, P.: Two and three dimensional magnetic reconnection observed in the Eulerian Lagrangian analysis of magnetohydrodynamics equations. Phys. Rev. E (3) \textbf{78}, 066315  (2008)

\bibitem{Sar}
Sarpkaya, T., Flow of non-Newtonian fluids in a magnetic field. AIChE. J. \textbf{7},  324-328 (1961)

\bibitem{Ser}
Sermange, M., Temam, R.: Some mathematical questions related to the
MHD equations. Comm. Pure Appl. Math. \textbf{36}, 635-664
(1983)

\bibitem{Simon}
Simon, J.: Compact sets in the space $\textbf{L}^p(0,T;B)$. Ann. Mat. Pura Appl. \textbf{146}, 65-96 (1986)


\bibitem{W}
Wolf, J.: Existence of weak solutions to the equations of non-stationary motion of non-Newtonian fluids with shear rate dependent viscosity. J. Math. Fluid Mech. \textbf{9}, 104-138 (2007)

\bibitem{Yan}
Yan, W.P.: Motion of compressible magnetic fluids in $\textbf{T}^3$. 
Electron. J. Differential Equations. \textbf{232}, 29 pp (2013)

\bibitem{Yan1}
Yan, W.P.: On weak-strong uniqueness property for full compressible magnetohydrodynamics flows. 
Cent. Eur. J. Math. \textbf{11},  2005-2019 (2013)

\end{thebibliography}


\end{document}